\documentclass[lettersize,journal]{IEEEtran}
\usepackage{amsmath,amssymb,amsfonts}
\usepackage{array}
\usepackage[caption=false,font=normalsize,labelfont=sf,textfont=sf]{subfig}
\usepackage{textcomp}
\usepackage{stfloats}
\usepackage{url}
\usepackage{verbatim}
\usepackage{graphicx}
\usepackage{cite}
\usepackage{mathtools}
\DeclareMathOperator*{\argmin}{argmin}
\allowdisplaybreaks

\usepackage{amsthm}
\newtheorem{proposition}{Proposition}
\usepackage{algorithm, algpseudocode}
\usepackage{multirow}
\usepackage{eurosym}
\usepackage[usenames,dvipsnames]{color}

\begin{document}

\title{
Stochastic Virtual Power Plant Dispatch via Temporally Aggregated Distributed Predictive Control with Performance Guarantees
}

\author{
Luca Santosuosso,~\IEEEmembership{Member,~IEEE,}
Fei Teng,~\IEEEmembership{Senior Member,~IEEE,}
and Sonja Wogrin,~\IEEEmembership{Senior Member,~IEEE}
\thanks{
Luca Santosuosso and Sonja Wogrin are with the Institute of Electricity Economics and Energy Innovation, Graz University of Technology, and the Research Center ENERGETIC, 8010 Graz, Austria (e-mail: luca.santosuosso@tugraz.at; wogrin@tugraz.at).

Fei Teng is with the Department of Electrical and Electronic Engineering, Imperial College London, SW7 2BU London, U.K. (e-mail: f.teng@imperial.ac.uk).}}

\markboth{}%
{Shell \MakeLowercase{\textit{et al.}}: A Sample Article Using IEEEtran.cls for IEEE Journals}

\maketitle

\begin{abstract}
This paper addresses the energy dispatch of a virtual power plant comprising renewable generation, energy storage, and thermal units under uncertainty in renewable output, energy prices, and energy demand. The nonlinear dynamics and multiple sources of uncertainty render traditional stochastic model predictive control (MPC) computationally intractable as the dispatch horizon, scenario set, and asset portfolio expand. To overcome this limitation, we propose a novel controller that seamlessly integrates MPC with time series aggregation and distributed optimization, simultaneously reducing the temporal, asset, and scenario dimensions of the problem. The resulting controller provides a rigorous performance guarantee through theoretically validated bounds on its approximation error, while leveraging dual information from previous MPC iterations to adaptively optimize the temporal aggregation. Numerical results show that the proposed controller reduces runtime by over 50\% relative to traditional stochastic MPC and, crucially, restores tractability where the full-scale dispatch model proves intractable.
\end{abstract}


\section{Introduction}
\subsection{Literature Review}
\IEEEPARstart{T}{he} global pursuit of carbon neutrality has accelerated the widespread integration of distributed energy resources (DERs),
including variable renewable energy sources (vRES), energy storage systems, and flexible loads, into modern power systems.
In this context, virtual power plants (VPPs) provide a structured framework for aggregating large populations of DERs into a unified portfolio
and coordinating their operation as a single entity \cite{bahloul2024residential},
enabling efficient participation in energy trading \cite{wei2023virtual} and the provision of grid services \cite{10443545}.

Although the potential of VPPs to enhance the flexibility of modern power systems is widely recognized \cite{naval2021virtual},
their operation remains a complex, multi-stage process encompassing long-term planning, short-term scheduling, and real-time dispatch \cite{naughton2021co}.
This paper focuses on the latter stage of this decision-making process.

The real-time energy dispatch of a VPP is commonly formulated as an optimal control problem \cite{bolzoni2020optimal},
in which operations are controlled at sub-hourly resolution to optimize a performance criterion, such as minimizing generation costs, while satisfying the technical constraints of DERs.
Model predictive control (MPC) has emerged as a widely adopted approach for this task, owing to its capacity to anticipate system dynamics, explicitly enforce constraints, and iteratively update control actions based on real-time measurements \cite{rosewater2019risk}.

As VPP operations are affected by multiple sources of uncertainty,
such as energy demand, vRES generation, and market prices,
previous research has focused on incorporating methods to handle stochasticity within MPC schemes.
Traditional deterministic MPC assumes perfect knowledge over the optimization horizon \cite{santosuosso2024stochastic}, an assumption that is often unrealistic.
To address this limitation, alternative control schemes have been investigated:
scenario-based stochastic MPC \cite{he2021stochastic}, which employs scenarios to represent potential uncertainty realizations;
robust MPC \cite{zhao2022distributed}, which presumes that uncertainty is confined within a predefined set;
and chance-constrained MPC \cite{shen2023stochastic}, which enforces probabilistic guarantees on constraint satisfaction under stochastic variations.
A comprehensive review of these methods is provided in \cite{Roald}.

Constraining uncertainty to a finite set, as in robust MPC, is often impractical for unbounded sources such as market prices and may result in excessively conservative solutions \cite{zhang2021robust}.
Chance-constrained MPC typically requires explicit reformulations to enforce probabilistic constraints, which often entail significant mathematical complexity and may also introduce conservatism.
Scenario-based stochastic MPC is therefore typically preferred due to its flexibility and its ability to represent diverse uncertainty sources.
However, as the size and heterogeneity of the VPP portfolio increase, the underlying uncertainty space expands accordingly,
potentially requiring a prohibitively large number of scenarios and rendering the MPC scheme computationally intractable.

Notably, beyond increasing with the number of scenarios,
i.e., the \textit{scenario dimension} of the control problem,
the computational complexity of stochastic MPC also grows with the number and heterogeneity of DERs within the VPP, 
i.e., its \textit{asset dimension} \cite{han2023distributed}.
Ensuring simultaneous scalability with respect to both dimensions is therefore essential to enable fast real-time dispatch of complex, large-scale VPPs \cite{santosuosso2025distributed}.

To this end, prior research has investigated mathematical decomposition methods
for deriving distributed MPC schemes that partition the centralized control problem into concurrently solved subproblems,
while preserving convergence to global optimality \cite{contreras2022distributed}.
Within this line of work, various distributed MPC schemes have been proposed \cite{chen2025real}, 
including those based on dual decomposition \cite{contreras2022distributed}, Douglas–Rachford splitting \cite{halvgaard2016distributed}, optimality condition decomposition (OCD) \cite{zhu2014decomposed}, analytical target cascading (ATC) \cite{zhao2022distributed}, approximate Newton direction methods \cite{baker2016distributed},
the alternating direction method of multipliers (ADMM) \cite{rezaei2022distributed},
as well as learning-based accelerated variants of these methods \cite{rajaei2025learning}.
Among these, ADMM and its variants \cite{Boyd} have gained widespread adoption due to their flexibility in handling diverse problem formulations and their well-established convergence properties \cite{feng2023update}.

Notably, while several studies have distributed computations across either the scenario \cite{lopez2018stochastic} or the asset \cite{jia2025admm} dimension of VPP optimization models independently,
few have addressed both simultaneously \cite{santosuosso2025scenario}.
Crucially, none of these studies addresses scalability across the third complexity dimension of the problem, namely its \textit{temporal dimension}, 
thereby leaving substantial potential for reducing the computational burden of MPC schemes unexploited.
This occurs because temporal decomposition decouples the intertemporal constraints of the original model, yielding subproblems that are inherently more challenging to coordinate toward convergence than those resulting from decompositions across assets or scenarios \cite{zhu2014decomposed}.

To alleviate the computational burden associated with the temporal dimension of energy optimization models,
time series aggregation (TSA) methods have emerged as a highly effective alternative to decomposition methods. 
By compacting the full-scale optimization model, TSA produces an aggregated model defined over a reduced set of representative time periods (or clusters),
while aiming to preserve the essential temporal dynamics of the original formulation \cite{teichgraeber2022time}.
Unlike decomposition methods, which aim at reducing computational complexity via parallelization, TSA seeks the same objective through a dimensionality reduction of the original optimization model.

Traditional \textit{a priori} TSA methods typically employ standard clustering techniques,
such as k-means \cite{zhang2020novel}, k-medoids \cite{schutz2018comparison}, and hierarchical clustering \cite{liu2017hierarchical},
to construct temporally aggregated models based solely on the statistical features of the input time series.
More recently, \textit{a posteriori} TSA has emerged as an alternative paradigm \cite{hilbers2023reducing},
augmenting the clustering procedure with features derived from the optimization model itself, such as insights from previous optimization runs,
to yield aggregations that more accurately capture the temporal dynamics of the original full-scale optimization model \cite{zhang2022model}.
Remarkably, a posteriori TSA has the potential to achieve \textbf{exact} temporal aggregation \cite{wogrin2023time},
yielding an aggregated model that faithfully replicates the output of its full-scale counterpart while delivering substantial computational savings.

While many studies have proposed TSA methods for large-scale energy optimization models \cite{sun2019data},
particularly for those involving intertemporal constraints that complicate the TSA procedure by requiring the preservation of temporal chronology \cite{pineda2018chronological},
research on integrating TSA within MPC schemes remains scarce \cite{deml2015role}.
Notably, none of the existing studies investigates the potential of emerging a posteriori TSA methods to reduce the temporal dimension in stochastic MPC.

Furthermore, since TSA yields an aggregated model that approximates its full-scale counterpart, it is essential in practical applications to provide \textit{performance guarantees} by quantifying the approximation error introduced through temporal aggregation \cite{santosuosso2025optimal}.
While formal guarantees have been established for linear programming (LP) \cite{teichgraeber2019clustering}, mixed-integer linear (MILP) and mixed-integer quadratic programming (MIQP) problems \cite{santosuosso2025we},
these theoretical results are not readily extendable to quadratic programming (QP) and quadratically constrained quadratic programming (QCQP) formulations,
which are widely employed to model typical nonlinear characteristics of VPPs \cite{zare2026systematic}, including quadratic cost functions \cite{babaei2019data} and quadratic emission constraints \cite{shui2024optimal}.

\subsection{Research Gaps and Contributions}

\begin{table*}[!t]
\caption{Comparison of this study with the relevant related literature}
\label{tab:literature_comparison}
\centering
\setlength{\tabcolsep}{5pt}
\begin{tabular}{|c|c|c|c|c|c|c|}
\hline
\textbf{Ref.} & \textbf{MPC} & \textbf{Problem} & \textbf{Decomposition} & \textbf{Distributed method} & \textbf{TSA} & \textbf{Performance-guaranteed}\\
\hline
\cite{he2021stochastic} & \checkmark (stochastic) & QP & $\times$ & $-$ & $\times$ & $\checkmark$\\
\cite{shen2023stochastic} & \checkmark (chance-constrained) & QP & $\times$ & $-$ & $\times$ & $\checkmark$\\
\cite{contreras2022distributed} & \checkmark (deterministic) & QP & Asset & Dual decomposition & $\times$ & $\checkmark$\\
\cite{zhao2022distributed} & \checkmark (robust) & MILP & Asset & ATC & $\times$ & $\times$\\
\cite{rezaei2022distributed} & \checkmark (stochastic) & MIQP & Asset & ADMM & $\times$ & $\times$\\
\cite{zhu2014decomposed} & \checkmark (stochastic) & QP & Scenario \& temporal & OCD & $\times$ & $\checkmark$\\
\cite{liu2017hierarchical} & $\times$ & LP & $\times$ & $-$ & \checkmark (a priori) & $\times$\\
\cite{zhang2022model} & $\times$ & MILP & $\times$ & $-$ & \checkmark (a posteriori) & $\times$\\
\cite{teichgraeber2019clustering} & $\times$ & LP & $\times$ & $-$ & \checkmark (a priori) & \checkmark\\
\cite{santosuosso2025we} & $\times$ & MILP \& MIQP & $\times$ & $-$ & \checkmark (a priori) & \checkmark\\
\cite{deml2015role} & \checkmark (deterministic) & MIQP & $\times$ & $-$ & \checkmark (a priori) & $\times$\\
This paper & \checkmark (stochastic) & QCQP & Asset \& scenario & ADMM & \checkmark (a posteriori) & \checkmark\\
\hline
\end{tabular}
\end{table*}

The literature review reveals the following research gaps:
\begin{itemize}
    \item Prior studies show that decomposition methods effectively enhance the scalability of energy dispatch problems along the scenario and asset dimensions \cite{contreras2022distributed,halvgaard2016distributed,zhu2014decomposed,zhao2022distributed,baker2016distributed,rezaei2022distributed},
    while TSA, particularly a posteriori TSA, has proven effective in reducing the temporal dimension without compromising solution accuracy \cite{zhang2020novel,schutz2018comparison,liu2017hierarchical,zhang2022model,wogrin2023time}.
    However, the systematic integration of these two complementary methodologies within MPC schemes remains largely unexplored, leaving significant potential for further reducing the computational complexity of real-time controllers.
    
    \item Existing research on performance guarantees for TSA methods has primarily focused on establishing bounds on the approximation error incurred by aggregated optimization models relative to their full-scale counterparts.
    To date, such guarantees have been derived only for LP \cite{teichgraeber2019clustering}, MILP, and MIQP problems \cite{santosuosso2025we},
    typically by constructing aggregated models in which the decision variables are averaged over the representative time periods.
    However, for QCQP problems commonly encountered in VPP dispatch \cite{naughton2021co},
    the simultaneous presence of quadratic constraints, time-varying parameters, and intertemporal constraints (e.g., due to storage and ramping) renders this averaging approach insufficient: as discussed in Section~\ref{sec:Methodology}, the previously established bounds no longer hold.
\end{itemize}
This paper seeks to address these research gaps through the following \textbf{key contributions}:
\begin{itemize}
    \item We formulate the energy dispatch of a VPP as a QCQP problem and propose a stochastic MPC scheme for its real-time solution.
    The stochastic MPC scheme is decomposed along both its asset and scenario dimensions using consensus ADMM, and temporally aggregated via TSA.
    This integrated approach, featuring a novel combination of MPC, TSA, and distributed optimization, simultaneously addresses all complexity dimensions of the dispatch problem, thereby maximizing computational efficiency.
    \item We demonstrate that the derived controller always provides a lower bound on the optimal objective function value of its full-scale counterpart,
    even in the presence of quadratic constraints, time-varying parameters, and intertemporal constraints such as storage and ramping.
    \item Leveraging the solution of the temporally aggregated and distributed controller,
    we propose an algorithm to systematically compute an upper bound on the optimal objective function value of the full-scale dispatch problem.
    The gap between the upper and lower bounds quantifies the approximation error incurred at each MPC iteration, thereby providing a rigorous performance guarantee.
    \item Although the derived bounds hold independently of the clustering technique employed for TSA,
    we show that incorporating dual information from the dispatch model into the clustering features yields a temporally aggregated controller that \textbf{exactly} reproduces the full-scale solution.
    Since such dual information is not available \emph{ex ante}, we propose a closed-loop, a posteriori TSA method leveraging past MPC iterations to estimate it.
\end{itemize}
Table~\ref{tab:literature_comparison} compares this study with the relevant related literature.

The remainder of the paper is organized as follows: Section~\ref{sec:Methodology} details the proposed methodology, Section~\ref{sec:Results} presents the results, and Section~\ref{sec:Conclusion} concludes the study.

\section{Methodology}
\label{sec:Methodology}
In the following, sets and vectors are denoted in boldface (e.g., $\boldsymbol{z}$).
The cardinality of a set is denoted by $|\cdot|$, while $\|\cdot\|_2$ denotes the Euclidean ($\ell_2$) norm.
The maximum and minimum elements of a set $\boldsymbol{\Phi}$ are denoted by $\max(\boldsymbol{\Phi})$ and $\min(\boldsymbol{\Phi})$, respectively.
The zero vector is denoted by $\mathbf{0}$. All sets are indexed starting from $0$.

\subsection{Overview of the Proposed Methodology}
In this section, we present the proposed methodology for real-time VPP dispatch under uncertainty.
The VPP comprises vRES, namely wind and solar, a set of energy storage systems $\boldsymbol{N}$ indexed by $n$, and a set of thermal power plants $\boldsymbol{G}$ indexed by $g$.
The dispatch problem is subject to various uncertainties, namely vRES generation, energy demand, and energy prices, captured by a set of scenarios $\boldsymbol{\Omega}$ indexed by $\omega$.
While vRES are non-dispatchable, storage and thermal units can be dispatched to maximize profit while meeting the VPP internal demand,
with any imbalances accommodated via non-supplied energy.
The VPP modeling is detailed in Subsection~\ref{subsec:Modeling}.

The dispatch problem is addressed using a stochastic MPC scheme, as discussed in Subsection~\ref{subsec:Full-Scale}.
At each time step $t \in \boldsymbol{T}$, the MPC scheme is solved over a prediction horizon $\boldsymbol{K}$, indexed by $k$, with cardinality $K \coloneqq |\boldsymbol{K}|$ and sampling time $\Delta$. 
Following the standard MPC paradigm, at each time step $t$, only the first control action (corresponding to $k=0$) is applied, after which the horizon is advanced by one time step.
To streamline the notation, we use the subscript $k$ to denote variables and parameters associated with time $t+k$, i.e., the $k$-th prediction step of the MPC iteration executed at time $t$.

Our goal is to optimize the real-time energy dispatch of a VPP while simultaneously reducing the computational complexity of the problem across its asset, scenario, and temporal dimensions.
To achieve this, we propose a novel integration of MPC, a posteriori TSA, and distributed optimization. 
Fig.~\ref{fig:distributed_MPC_TSA} presents an overview of the proposed methodology.

The a posteriori TSA module exploits dual information from the previous MPC iteration to aggregate the full-scale model across $\boldsymbol{R}$ representative time periods, indexed by $r$.
The resulting temporally aggregated model is then simultaneously decomposed across assets and scenarios, yielding the proposed \textbf{temporally aggregated and distributed stochastic MPC scheme}, as detailed in Subsection~\ref{subsec:AggregatedDistributed}.
The temporally aggregated stochastic MPC scheme and a posteriori TSA method are detailed in Subsections~\ref{subsec:Aggregated} and \ref{subsec:APosterioriTSA}, respectively.

As illustrated in Fig.~\ref{fig:distributed_MPC_TSA}, the controller delivers a rigorous \textbf{performance guarantee} via iteratively refined upper and lower bounds on the full-scale optimal objective function value.
The gap between these bounds enables the decision-maker to assess the controller performance without requiring the full-scale model solution, as discussed in Subsection~\ref{subsec:TheoreticalResult}.

\begin{figure*}[t]
  \centering
  \includegraphics[width=\textwidth]{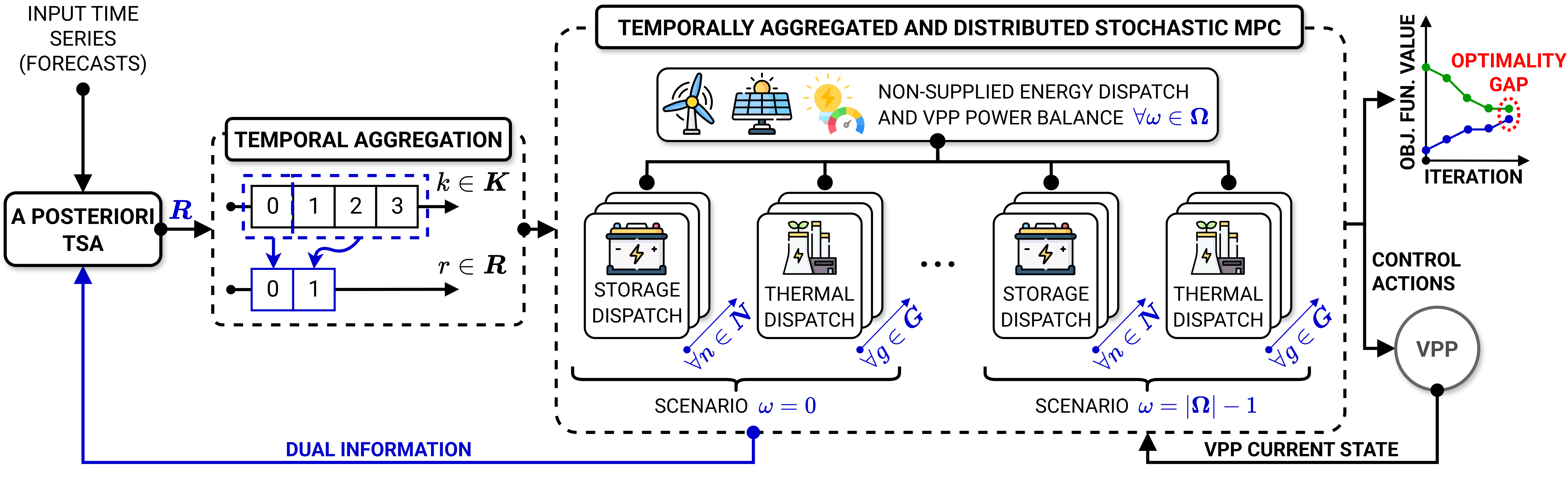}
  \caption{Illustration of the proposed methodology for real-time energy dispatch of a VPP under uncertainty.}
  \label{fig:distributed_MPC_TSA}
\end{figure*}

\subsection{Modeling}
\label{subsec:Modeling}

The VPP dispatch is modeled as follows.
For each scenario $\omega$ and time step $k$, the forecasts of vRES generation, demand, and energy prices are denoted by $P^{\mathrm{vRES}}_{\omega,k}$, $D_{\omega,k}$, and $\pi_{\omega,k}$, respectively.
The energy output of the VPP in scenario $\omega$ at time $k$ is denoted by $e_{\omega,k}$, while $e^{\mathrm{ns}}_{\omega,k}$ is the non-supplied energy demand.
For each storage unit $n$, the state of charge at time $k$ in scenario $\omega$ is denoted by $e^{\mathrm{s}}_{n,\omega,k}$,
with charging and discharging powers denoted by $p^{\mathrm{c}}_{n,\omega,k}$ and $p^{\mathrm{d}}_{n,\omega,k}$, respectively.
The maximum charging and discharging power limits are denoted by $\overline{P}^{\mathrm{c}}_n$ and $\overline{P}^{\mathrm{d}}_n$, respectively.
Charging and discharging efficiencies are denoted by $\eta^{\mathrm{c}}_n$ and $\eta^{\mathrm{d}}_n$.
The initial state of charge is $E^{\mathrm{s,0}}_{n}$,
while the lower and upper bounds on the state of charge are $\underline{E}^{\mathrm{s}}_n$ and $\overline{E}^{\mathrm{s}}_n$, respectively.

The dynamics of each energy storage system are given by
\begin{align}
    e^{\mathrm{s}}_{n,\omega,k+1} & = e^{\mathrm{s}}_{n,\omega,k} + \left(\eta^{\mathrm{c}}_n p^{\mathrm{c}}_{n,\omega,k} - \eta^{\mathrm{d}}_n p^{\mathrm{d}}_{n,\omega,k}\right) \Delta, \nonumber\\
    & \qquad\qquad\qquad\qquad \forall n, \forall \omega, \forall k \in \boldsymbol{K} \setminus \{K-1\}, \label{model:storage_soc}\\
    e^{\mathrm{s}}_{n,\omega,0} & = E^{\mathrm{s,0}}_{n}, \; \forall n, \forall \omega, \label{model:storage_soc_init}\\
    \underline{E}^{\mathrm{s}}_n & \leq e^{\mathrm{s}}_{n,\omega,k} \leq \overline{E}^{\mathrm{s}}_n, \; \forall n, \forall \omega, \forall k. \label{model:storage_soc_bounds}
\end{align}
The charging and discharging powers are constrained by
\begin{align}
    0 & \leq p^\mathrm{c}_{n,\omega,k} \leq \overline{P}^\mathrm{c}_n, \; \forall n, \forall \omega, \forall k, \label{model:storage_char_bounds}\\
    0 & \leq p^\mathrm{d}_{n,\omega,k} \leq \overline{P}^\mathrm{d}_n, \; \forall n, \forall \omega, \forall k. \label{model:storage_dischar_bounds}
\end{align}

The power output of each thermal unit $g$, denoted $p^{\mathrm{th}}_{g,\omega,k}$, is constrained between its lower ($\underline{P}^\mathrm{th}_g$) and upper ($\overline{P}^\mathrm{th}_g$) limits as
\begin{equation}
    \underline{P}^\mathrm{th}_g \leq p^\mathrm{th}_{g,\omega,k} \leq \overline{P}^\mathrm{th}_g, \; \forall g, \forall \omega, \forall k. \label{model:thr_bounds}
\end{equation}
Moreover, the thermal power generation of the VPP is subject to the following quadratic emission constraints \cite{moghimi2013stochastic}:
\begin{equation}
    \alpha_g \left(p^\mathrm{th}_{g,\omega,k}\right)^2 + \beta_g \, p^\mathrm{th}_{g,\omega,k} \leq L^{\mathrm{CO}_2}_g, \; \forall g, \forall \omega, \forall k, \label{model:thr_emission}
\end{equation}
where $L^{\mathrm{CO}_2}_g$ is the emission limit of generator $g$,
while $\alpha_g$ and $\beta_g$ are the coefficients mapping its output to emissions.

Each thermal unit is subject to its ramp limit $R^\mathrm{th}_g$ as
\begin{align}
    p^\mathrm{th}_{g,\omega,k+1} - p^\mathrm{th}_{g,\omega,k} & \leq R^\mathrm{th}_g, \; \forall g, \forall \omega, \forall k \in \boldsymbol{K}\setminus\{K-1\}, \label{model:thr_ramp1}\\
    p^\mathrm{th}_{g,\omega,k} - p^\mathrm{th}_{g,\omega,k+1} & \leq R^\mathrm{th}_g, \; \forall g, \forall \omega, \forall k \in \boldsymbol{K}\setminus\{K-1\}. \label{model:thr_ramp2}
\end{align}

\subsection{Full-Scale Controller}
\label{subsec:Full-Scale}

We formulate the VPP energy dispatch as a stochastic MPC problem.
Let $\boldsymbol{u}_{0}$ denote the vector of control actions computed for $k\! = \!0$.
We define the set of decision variables as $\boldsymbol{z} \!\coloneqq\! \big\{ \boldsymbol{u}_{0},$ $e_{\omega,k}, e^{\mathrm{ns}}_{\omega,k}, e^{\mathrm{s}}_{n,\omega,k}, p^\mathrm{th}_{g,\omega,k}, p^\mathrm{c}_{n,\omega,k}, p^\mathrm{d}_{n,\omega,k} \big\}_{n \in \boldsymbol{N}, g \in \boldsymbol{G}, \omega \in \boldsymbol{\Omega}, k \in \boldsymbol{K}}$.

Let $C^{\mathrm{c}}_n$ and $C^{\mathrm{d}}_n$ denote the charging and discharging costs of storage unit $n$, respectively,
and let $C^{\mathrm{th}}_g$ represent the generation cost of thermal unit $g$.
The penalty associated with non-supplied energy is denoted by $C^{\mathrm{ns}}$.
Furthermore, a reference state $E^{\mathrm{ref}}_n$ is specified for each storage unit $n$;
deviations from this reference are penalized by $C^{\mathrm{ref}}_n$.

Let $\phi_\omega$ denote the probability associated with scenario $\omega \in \boldsymbol{\Omega}$.
The goal of the stochastic dispatch problem is to minimize the objective function $F(\boldsymbol{z})$, defined as
\begin{align}
\label{FS_SMPC:obj}
F(\boldsymbol{z}) \coloneqq &
\sum_{\omega \in \boldsymbol{\Omega}} \phi_{\omega} \sum_{k \in \boldsymbol{K}} \left(
- \pi_{\omega,k} \, e_{\omega,k} + C^{\mathrm{ns}} e^{\mathrm{ns}}_{\omega,k} \right) \nonumber\\
& + \sum_{\omega \in \boldsymbol{\Omega}} \phi_{\omega} \sum_{g \in \boldsymbol{G}} \sum_{k \in \boldsymbol{K}} C^{\mathrm{th}}_g p^\mathrm{th}_{g,\omega,k} \Delta \nonumber\\
& + \sum_{\omega \in \boldsymbol{\Omega}} \phi_{\omega} \sum_{n \in \boldsymbol{N}} \sum_{k \in \boldsymbol{K}} \left(C^{\mathrm{c}}_n p^\mathrm{c}_{n,\omega,k} + C^{\mathrm{d}}_n p^\mathrm{d}_{n,\omega,k}\right) \Delta \nonumber\\
& + \sum_{\omega \in \boldsymbol{\Omega}} \phi_{\omega} \sum_{n \in \boldsymbol{N}} \sum_{k \in \boldsymbol{K}} C^{\mathrm{ref}}_n \left(e^{\mathrm{s}}_{n,\omega,k} - E^{\mathrm{ref}}_n\right)^2.
\end{align}
In \eqref{FS_SMPC:obj}, the linear terms represent the expected economic performance of the VPP by accounting for market revenues, thermal generation costs, storage charging and discharging costs, and penalties for non-supplied energy,
whereas the quadratic term penalizes expected deviations of the storage states from their reference values.

We assume that the penalty for non-supplied energy exceeds the maximum market price across all scenarios and time steps,
i.e., $C^{\mathrm{ns}} > \max_{\omega \in \boldsymbol{\Omega}, k \in \boldsymbol{K}} \left\{\pi_{\omega,k}\right\}$.

At each time $t \in \boldsymbol{T}$, the \textbf{full-scale stochastic MPC scheme} solves the following QCQP problem over the horizon $\boldsymbol{K}$:
\begin{subequations}
\label{FS_SMPC}
\begin{align}
\min_{\boldsymbol{z}} \quad & F\left(\boldsymbol{z}\right)\\
\textrm{s.t.} \quad & e_{\omega, k} = P^{\mathrm{vRES}}_{\omega,k} \, \Delta - D_{\omega,k} + \sum_{n \in \boldsymbol{N}} \left(p^{\mathrm{d}}_{n,\omega,k} - p^{\mathrm{c}}_{n,\omega,k} \right) \Delta \nonumber\\
& \qquad\quad + \sum_{g \in \boldsymbol{G}} p^\mathrm{th}_{g,\omega,k} \, \Delta + e^{\mathrm{ns}}_{\omega,k}, \; \forall \omega, \forall k, \label{FS_SMPC:bal}\\
& \eqref{model:storage_soc}-\eqref{model:thr_ramp2}, \nonumber\\
& \left[ e^{\mathrm{ns}}_{\omega,0}, p^\mathrm{th}_{g,\omega,0}, p^{\mathrm{c}}_{n,\omega,0}, p^{\mathrm{d}}_{n,\omega,0} \right] = \boldsymbol{u}_{0}, \; \forall n, \forall g, \forall \omega, \label{FS_SMPC:non_ant}\\
& \boldsymbol{z} \geq \boldsymbol{0}. \label{FS_SMPC:pos_vars}
\end{align}
\end{subequations}
In \eqref{FS_SMPC}, the constraints \eqref{FS_SMPC:bal} enforce the energy balance within the VPP, 
while \eqref{FS_SMPC:non_ant} are the non-anticipativity constraints.

\subsection{Temporally Aggregated Controller}
\label{subsec:Aggregated}

To alleviate computational  burden, we employ TSA to construct a temporally aggregated counterpart of \eqref{FS_SMPC}, defined over $\boldsymbol{R}$ representative time periods, with $R \coloneqq |\boldsymbol{R}|$.
When $R \ll K$, solving the aggregated model in place of its full-scale counterpart yields a significant computational advantage.

Let $\boldsymbol{K}_r \subseteq \boldsymbol{K}$ denote the set of time steps $k$ assigned to cluster $r \in \boldsymbol{R}$, with cardinality $K_r \coloneqq |\boldsymbol{K}_r|$.
We collect the decision variables of the aggregated model in $\boldsymbol{\bar{z}} \coloneqq \{ \boldsymbol{\bar{u}}_{0}, $ $\bar{e}_{\omega,r}, \bar{e}^{\mathrm{ns}}_{\omega,r}, \bar{e}^{\mathrm{s}}_{n,\omega,r}, \bar{p}^\mathrm{th}_{g,\omega,r}, \bar{p}^\mathrm{c}_{n,\omega,r}, \bar{p}^\mathrm{d}_{n,\omega,r} \}_{n \in \boldsymbol{N}, g \in \boldsymbol{G}, \omega \in \boldsymbol{\Omega}, r \in \boldsymbol{R}}$.

We define the aggregated counterpart of $F(\boldsymbol{z})$ in \eqref{FS_SMPC:obj} as
\begin{align}
\label{AGG_SMPC:obj}
\bar{F}(\boldsymbol{\bar{z}}) \coloneqq &
\sum_{\omega \in \boldsymbol{\Omega}} \phi_{\omega} \sum_{r \in \boldsymbol{R}} K_r \left(
- \max_{k \in \boldsymbol{K}_r}\left\{\pi_{\omega,k}\right\} \, \bar{e}_{\omega,r} + C^{\mathrm{ns}} \bar{e}^{\mathrm{ns}}_{\omega,r} \right) \nonumber\\
& + \sum_{\omega \in \boldsymbol{\Omega}} \phi_{\omega} \sum_{g \in \boldsymbol{G}} \sum_{r \in \boldsymbol{R}} K_r C^{\mathrm{th}}_g \bar{p}^\mathrm{th}_{g,\omega,r} \Delta \nonumber\\
& + \sum_{\omega \in \boldsymbol{\Omega}} \phi_{\omega} \sum_{n \in \boldsymbol{N}} \sum_{r \in \boldsymbol{R}} K_r \left( C^{\mathrm{c}}_n \bar{p}^\mathrm{c}_{n,\omega,r} + C^{\mathrm{d}}_n \bar{p}^\mathrm{d}_{n,\omega,r} \right) \Delta \nonumber\\
& + \sum_{\omega \in \boldsymbol{\Omega}} \phi_{\omega} \sum_{n \in \boldsymbol{N}} \sum_{r \in \boldsymbol{R}} C^{\mathrm{ref}}_n \left(\bar{e}^{\mathrm{s}}_{n,\omega,r} - E^{\mathrm{ref}}_n\right)^2.
\end{align}

At each time $t \in \boldsymbol{T}$, the \textbf{temporally aggregated stochastic MPC scheme} solves the following QCQP problem over $\boldsymbol{R}$:
\begin{subequations}
\label{AGG_SMPC}
\begin{align}
\min_{\boldsymbol{\bar{z}}} \quad & \bar{F}\left(\boldsymbol{\bar{z}}\right)\\
\textrm{s.t.} \quad & \bar{e}_{\omega, r} = \sum_{k \in \boldsymbol{K}_r} \! \frac{P^{\mathrm{vRES}}_{\omega,k} \Delta - D_{\omega,k}}{K_r} + \sum_{g \in \boldsymbol{G}} \bar{p}^\mathrm{th}_{g,\omega,r} \Delta + \bar{e}^{\mathrm{ns}}_{\omega,r} \nonumber\\
& \qquad\quad + \sum_{n \in \boldsymbol{N}} \left(\bar{p}^{\mathrm{d}}_{n,\omega,r} - \bar{p}^{\mathrm{c}}_{n,\omega,r} \right) \Delta, \; \forall \omega, \forall r, \label{AGG_SMPC:bal}\\
& \bar{e}^{\mathrm{s}}_{n,\omega,r+1} = \bar{e}^{\mathrm{s}}_{n,\omega,r} + K_r \left(\eta^{\mathrm{c}}_n \bar{p}^{\mathrm{c}}_{n,\omega,r} - \eta^{\mathrm{d}}_n \bar{p}^{\mathrm{d}}_{n,\omega,r}\right) \Delta, \nonumber\\
& \qquad\qquad\qquad\qquad\; \forall n, \forall \omega, \forall r \in \boldsymbol{R} \setminus \{R-1\}, \label{AGG_SMPC:storage_soc}\\
& \bar{e}^{\mathrm{s}}_{n,\omega,0} = E^{\mathrm{s,0}}_{n}, \; \forall n, \forall \omega, \label{AGG_SMPC:storage_soc_init}\\
& \underline{E}^{\mathrm{s}}_n \leq \bar{e}^{\mathrm{s}}_{n,\omega,r} \leq \overline{E}^{\mathrm{s}}_n, \; \forall n, \forall \omega, \forall r, \label{AGG_SMPC:storage_soc_bounds}\\
& 0 \leq \bar{p}^\mathrm{c}_{n,\omega,r} \leq \overline{P}^\mathrm{c}_n, \; \forall n, \forall \omega, \forall r, \label{AGG_SMPC:storage_char_bounds}\\
& 0 \leq \bar{p}^\mathrm{d}_{n,\omega,r} \leq \overline{P}^\mathrm{d}_n, \; \forall n, \forall \omega, \forall r, \label{AGG_SMPC:storage_dischar_bounds}\\
& \underline{P}^\mathrm{th}_g \leq \bar{p}^\mathrm{th}_{g,\omega,r} \leq \overline{P}^\mathrm{th}_g, \; \forall g, \forall \omega, \forall r, \label{AGG_SMPC:thr_bounds}\\
& \alpha_g K_r \left(\bar{p}^\mathrm{th}_{g,\omega,r}\right)^2 - \alpha_g (K_r - 1) \left(\overline{P}^\mathrm{th}_g\right)^2 \nonumber\\
& \qquad\qquad\qquad + \beta_g \, \bar{p}^\mathrm{th}_{g,\omega,r} \leq L^{\mathrm{CO}_2}_g, \; \forall g, \forall \omega, \forall r, \label{AGG_SMPC:thr_emission}\\
& \bar{p}^\mathrm{th}_{g,\omega,r+1} - \bar{p}^\mathrm{th}_{g,\omega,r} \, K_r \leq R^\mathrm{th}_g + R^\mathrm{th}_g \, \frac{K_{r+1} - 1}{2}, \nonumber\\
& \qquad\qquad\qquad\qquad\; \forall g, \forall \omega, \forall r \in \boldsymbol{R} \setminus \{R-1\}, \label{AGG_SMPC:thr_ramp1}\\
& \bar{p}^\mathrm{th}_{g,\omega,r} - \bar{p}^\mathrm{th}_{g,\omega,r+1} \, K_{r+1} \leq R^\mathrm{th}_g + R^\mathrm{th}_g \, \frac{K_r - 1}{2}, \nonumber\\
& \qquad\qquad\qquad\qquad\; \forall g, \forall \omega, \forall r \in \boldsymbol{R} \setminus \{R-1\}, \label{AGG_SMPC:thr_ramp2}\\
& \left[ \bar{e}^{\mathrm{ns}}_{\omega,0}, \bar{p}^\mathrm{th}_{g,\omega,0}, \bar{p}^{\mathrm{c}}_{n,\omega,0}, \bar{p}^{\mathrm{d}}_{n,\omega,0} \right] = \boldsymbol{\bar{u}}_{0}, \; \forall n, \forall g, \forall \omega, \label{AGG_SMPC:non_ant}\\
& \boldsymbol{\bar{z}} \geq \boldsymbol{0}. \label{AGG_SMPC:pos_vars}
\end{align}
\end{subequations}
Here, \eqref{AGG_SMPC:bal}, \eqref{AGG_SMPC:storage_soc}--\eqref{AGG_SMPC:thr_ramp2}, \eqref{AGG_SMPC:non_ant}, and \eqref{AGG_SMPC:pos_vars} are the aggregated counterparts of \eqref{FS_SMPC:bal}, \eqref{model:storage_soc}--\eqref{model:thr_ramp2}, \eqref{FS_SMPC:non_ant}, and \eqref{FS_SMPC:pos_vars}, respectively.

Notably, whereas conventional TSA methods average the input time series over each cluster \cite{teichgraeber2022time},
we introduce a novel formulation that instead considers the cluster-wise maximum price in \eqref{AGG_SMPC:obj}
and augments the quadratic constraints \eqref{AGG_SMPC:thr_emission} and the intertemporal constraints \eqref{AGG_SMPC:storage_soc}, \eqref{AGG_SMPC:thr_ramp1}, and \eqref{AGG_SMPC:thr_ramp2} with cluster-specific terms.
In particular, compared with their original full-scale formulation, the aggregated emission constraints \eqref{AGG_SMPC:thr_emission} are augmented with cluster-specific correction terms, $- \alpha_g (K_r - 1) \left(\overline{P}^\mathrm{th}_g\right)^2$, to compensate for the loss of visibility into intra-cluster variations caused by the temporal aggregation.
These correction terms ensure that the image of any full-scale feasible trajectory under the temporal mapping is contained within the aggregated feasible domain, as formally established in Proposition~\ref{prop:main_result}.

\subsection{Closed-Loop A Posteriori Time Series Aggregation}
\label{subsec:APosterioriTSA}

Once the temporally aggregated model \eqref{AGG_SMPC} is constructed,
a clustering technique is required to extract representative time periods from the original prediction horizon $\boldsymbol{K}$.
Owing to the presence of intertemporal constraints in \eqref{FS_SMPC}, conventional clustering techniques (e.g., k-means) cannot be directly employed, as they disregard temporal chronology.
To address this issue, we employ a \emph{sliding window clustering} technique.

Let $\{\boldsymbol{a}_k\}_{k \in \boldsymbol{K}}$ be the set of clustering features used for TSA.
This technique evaluates the similarity between consecutive elements of $\{\boldsymbol{a}_k\}$, assigning $k$ to the current cluster $r$ if
\begin{equation}
\label{TSACondition}
\|\boldsymbol{a}_k - \boldsymbol{\mu}_r\|_2 \leq \zeta,
\end{equation}
where $\zeta$ is a user-defined similarity threshold and $\boldsymbol{\mu}_r$ denotes the centroid of cluster $r$, computed as
$\boldsymbol{\mu}_r \coloneqq \frac{1}{K_r} \sum_{k \in \boldsymbol{K}_r} \boldsymbol{a}_k$.
If condition \eqref{TSACondition} is not satisfied for $\boldsymbol{a}_k$, the current cluster is closed and a new cluster is initialized starting at $k$.

Fig.~\ref{fig:sliding_window_clustering} illustrates the sliding window clustering technique.
Starting from an initial cluster (e.g., containing $k-1$ and $k$), the cluster is extended to include $k+1$ only if condition \eqref{TSACondition} is satisfied for $\boldsymbol{a}_{k+1}$, thereby ensuring temporal coherence.

\begin{figure}[t]
    \centering
    \includegraphics[scale=0.08]{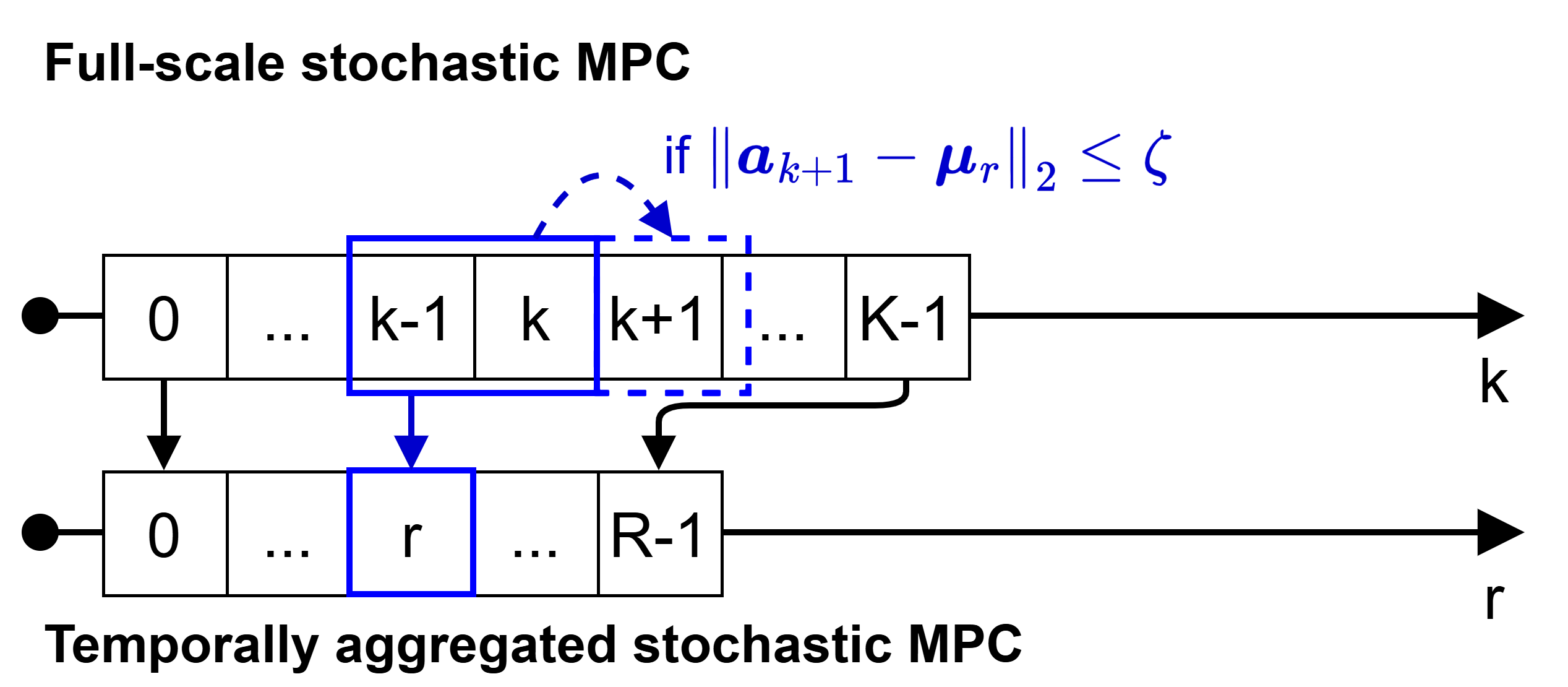}
    \caption{Illustration of the sliding window clustering technique.}
    \label{fig:sliding_window_clustering}
\end{figure}

The features used for TSA directly affect the fidelity of the aggregated model.
In particular, \cite{wogrin2023time} demonstrates that if the active constraints of the full-scale model were known,
an aggregated model could be constructed that exactly reproduces the optimal objective function value of the full-scale model.

Although such information is unavailable \textit{ex ante}, MPC offers a unique opportunity to exploit it retrospectively,
as the same optimization problem is repeatedly solved over overlapping horizons, differing only by the newly appended terminal time step and the set of scenarios considered at each MPC iteration.
Consequently, the outputs from previous MPC runs can inform TSA in a closed-loop manner.

Motivated by this insight, we propose a closed-loop a posteriori TSA method integrated with MPC.
At each iteration $t$, the dual information from the MPC problem solved at time $t-1$ is used as a feature in the sliding window clustering technique to identify the representative periods $\boldsymbol{R}$ (see Fig.~\ref{fig:distributed_MPC_TSA}).
Specifically, we employ the scenario-averaged marginal costs computed at the previous MPC iteration as estimates of the true marginal costs of the dispatch problem,
i.e., the dual variables associated with the energy balance constraints \eqref{FS_SMPC:bal}, which, as shown in Section \ref{sec:Results}, serve as an accurate proxy for identifying the active constraints of \eqref{FS_SMPC}.
Two singleton clusters are enforced at the first and last time steps of $\boldsymbol{K}$,
reflecting the absence of dual information at the terminal step and the critical role of the initial step, whose control actions are implemented at time $t$.
By exploiting the dual information from the previous MPC iteration, our method estimates the active constraints of \eqref{FS_SMPC}, directly targeting exact temporal aggregation, which is typically unattainable with traditional a priori TSA methods.

\subsection{Temporally Aggregated and Distributed Controller}
\label{subsec:AggregatedDistributed}

To further reduce the computational complexity arising from the asset and scenario dimensions, we decompose the temporally aggregated model \eqref{AGG_SMPC} via consensus ADMM \cite{Boyd}.

The aggregated model \eqref{AGG_SMPC} features three sets of coupling constraints: 
intertemporal coupling induced by the storage \eqref{AGG_SMPC:storage_soc} and ramping \eqref{AGG_SMPC:thr_ramp1}–\eqref{AGG_SMPC:thr_ramp2} constraints,
asset coupling arising from the energy balance constraints \eqref{AGG_SMPC:bal},
and scenario coupling induced by the non-anticipativity constraints \eqref{AGG_SMPC:non_ant}. 
All coupling constraints in \eqref{AGG_SMPC} involve three sets of \textit{global variables}: 
$\boldsymbol{\bar{p}}^{\boldsymbol{\mathrm{c}}}_{n,\omega} \coloneqq \left[\bar{p}^{\mathrm{c}}_{n,\omega,0}, ..., \bar{p}^{\mathrm{c}}_{n,\omega,r}, ..., \bar{p}^{\mathrm{c}}_{n,\omega,R-1}\right]^\top$, 
$\boldsymbol{\bar{p}}^{\boldsymbol{\mathrm{d}}}_{n,\omega} \coloneqq \left[\bar{p}^{\mathrm{d}}_{n,\omega,0}, ..., \bar{p}^{\mathrm{d}}_{n,\omega,r}, ..., \bar{p}^{\mathrm{d}}_{n,\omega,R-1}\right]^\top$, 
and
$\boldsymbol{\bar{p}}^{\boldsymbol{\mathrm{th}}}_{g,\omega} \coloneqq \left[\bar{p}^{\mathrm{th}}_{g,\omega,0}, ..., \bar{p}^{\mathrm{th}}_{g,\omega,r}, ..., \bar{p}^{\mathrm{th}}_{g,\omega,R-1}\right]^\top$.
We relax the coupling constraints in \eqref{AGG_SMPC} by introducing copies of these global variables.

Let
{\small $\left\{\boldsymbol{\hat{p}}^{\boldsymbol{\mathrm{c}}}_{n,\omega}, \boldsymbol{\hat{p}}^{\boldsymbol{\mathrm{d}}}_{n,\omega}, \boldsymbol{\hat{p}}^{\boldsymbol{\mathrm{th}}}_{g,\omega}\right\}$}
and
{\small $\left\{\boldsymbol{\tilde{p}}^{\boldsymbol{\mathrm{c}}}_{n,\omega}, \boldsymbol{\tilde{p}}^{\boldsymbol{\mathrm{d}}}_{n,\omega}, \boldsymbol{\tilde{p}}^{\boldsymbol{\mathrm{th}}}_{g,\omega}\right\}$}
denote two distinct sets of copies of the global variables in \eqref{AGG_SMPC}.
By introducing these copies, we reformulate the temporally aggregated model \eqref{AGG_SMPC} as the following consensus problem:
\begin{subequations}
\label{eq:consensus_ADMM_prob}
\begin{align}
\min_{\boldsymbol{\bar{z}}} \quad & \bar{F}(\boldsymbol{\bar{z}}) \\
\text{s.t.} \quad & \boldsymbol{z}^{\boldsymbol{\mathrm{b}}} \in \boldsymbol{\Xi}(\boldsymbol{\beta}), \label{eq:CADMM_bal}\\
& \boldsymbol{z}^{\boldsymbol{\mathrm{s}}}_{n,\omega} \in \boldsymbol{\Gamma}^{\boldsymbol{\mathrm{s}}}_{n,\omega}\left(\boldsymbol{\theta}^{\boldsymbol{\mathrm{s}}}_{n,\omega}\right), \; \forall n, \forall \omega, \label{eq:CADMM_storage}\\
& \boldsymbol{\tilde{p}}^{\boldsymbol{\mathrm{th}}}_{g,\omega} \in \boldsymbol{\Gamma}^{\boldsymbol{\mathrm{th}}}_{g,\omega}\left(\boldsymbol{\theta}^{\boldsymbol{\mathrm{th}}}_{g,\omega}\right), \; \forall g, \forall \omega, \label{eq:CADMM_thermal}\\
& \boldsymbol{\bar{p}}^{\boldsymbol{\mathrm{s}}}_{n,\omega} = \boldsymbol{\hat{p}}^{\boldsymbol{\mathrm{s}}}_{n,\omega} \; : \; \boldsymbol{\hat{\lambda}}^{\boldsymbol{\mathrm{s}}}_{n,\omega}, \; \forall n, \forall \omega, \label{eq:CADMM_consensus1}\\
& \boldsymbol{\bar{p}}^{\boldsymbol{\mathrm{s}}}_{n,\omega} = \boldsymbol{\tilde{p}}^{\boldsymbol{\mathrm{s}}}_{n,\omega} \; : \; \boldsymbol{\tilde{\lambda}}^{\boldsymbol{\mathrm{s}}}_{n,\omega}, \; \forall n, \forall \omega, \label{eq:CADMM_consensus2}\\
& \boldsymbol{\bar{p}}^{\boldsymbol{\mathrm{th}}}_{g,\omega} = \boldsymbol{\hat{p}}^{\boldsymbol{\mathrm{th}}}_{g,\omega} \; : \; \boldsymbol{\hat{\lambda}}^{\boldsymbol{\mathrm{th}}}_{g,\omega}, \; \forall g, \forall \omega, \label{eq:CADMM_consensus3}\\
& \boldsymbol{\bar{p}}^{\boldsymbol{\mathrm{th}}}_{g,\omega} = \boldsymbol{\tilde{p}}^{\boldsymbol{\mathrm{th}}}_{g,\omega} \; : \; \boldsymbol{\tilde{\lambda}}^{\boldsymbol{\mathrm{th}}}_{g,\omega}, \; \forall g, \forall \omega, \label{eq:CADMM_consensus4}
\end{align}
\end{subequations}
where
$\boldsymbol{\hat{\lambda}}^{\boldsymbol{\mathrm{s}}}_{n,\omega} \!\coloneqq\! \left[\boldsymbol{\hat{\lambda}}^{\boldsymbol{\mathrm{c}}}_{n,\omega}, \boldsymbol{\hat{\lambda}}^{\boldsymbol{\mathrm{d}}}_{n,\omega}\right]^\top$,
$\boldsymbol{\tilde{\lambda}}^{\boldsymbol{\mathrm{s}}}_{n,\omega} \!\coloneqq\! \left[\boldsymbol{\tilde{\lambda}}^{\boldsymbol{\mathrm{c}}}_{n,\omega}, \boldsymbol{\tilde{\lambda}}^{\boldsymbol{\mathrm{d}}}_{n,\omega}\right]^\top$,
$\boldsymbol{\tilde{p}}^{\boldsymbol{\mathrm{s}}}_{n,\omega} \!\!\coloneqq\! [\boldsymbol{\tilde{p}}^{\boldsymbol{\mathrm{c}}}_{n,\omega}, \boldsymbol{\tilde{p}}^{\boldsymbol{\mathrm{d}}}_{n,\omega}]^\top$,
$\boldsymbol{\hat{p}}^{\boldsymbol{\mathrm{s}}}_{n,\omega} \!\!\coloneqq\! [\boldsymbol{\hat{p}}^{\boldsymbol{\mathrm{c}}}_{n,\omega}, \boldsymbol{\hat{p}}^{\boldsymbol{\mathrm{d}}}_{n,\omega}]^\top$,
$\boldsymbol{\bar{p}}^{\boldsymbol{\mathrm{s}}}_{n,\omega} \!\!\coloneqq\! [\boldsymbol{\bar{p}}^{\boldsymbol{\mathrm{c}}}_{n,\omega}, \boldsymbol{\bar{p}}^{\boldsymbol{\mathrm{d}}}_{n,\omega}]^\top$,
$\boldsymbol{z}^{\boldsymbol{\mathrm{b}}} \coloneqq \big\{\boldsymbol{\bar{u}}_{0}, \bar{e}_{\omega,r}, \bar{e}^{\mathrm{ns}}_{\omega,r}, \boldsymbol{\hat{p}}^{\boldsymbol{\mathrm{c}}}_{n,\omega}, \boldsymbol{\hat{p}}^{\boldsymbol{\mathrm{d}}}_{n,\omega}, \boldsymbol{\hat{p}}^{\boldsymbol{\mathrm{th}}}_{g,\omega}\big\}_{n \in \boldsymbol{N}, g \in \boldsymbol{G}, \omega \in \boldsymbol{\Omega}, r \in \boldsymbol{R}}$,
and
$\boldsymbol{z}^{\boldsymbol{\mathrm{s}}}_{n,\omega} \coloneqq \big\{\bar{e}^{\mathrm{s}}_{n,\omega,r}, \boldsymbol{\tilde{p}}^{\boldsymbol{\mathrm{c}}}_{n,\omega}, \boldsymbol{\tilde{p}}^{\boldsymbol{\mathrm{d}}}_{n,\omega}\big\}_{r \in \boldsymbol{R}}$.

This reformulation introduces three sets of temporally aggregated subproblems:
\begin{itemize}
    \item Subproblem \eqref{STEPI_1}, corresponding to \eqref{eq:CADMM_bal}, with local variables
    $\boldsymbol{z}^{\boldsymbol{\mathrm{b}}}$, and feasible set $\boldsymbol{\Xi}(\boldsymbol{\beta})$, defined by \eqref{AGG_SMPC:bal}, \eqref{AGG_SMPC:storage_char_bounds}--\eqref{AGG_SMPC:thr_bounds}, \eqref{AGG_SMPC:non_ant} and \eqref{AGG_SMPC:pos_vars}, with parameters $\boldsymbol{\beta}$.
    \item $|\boldsymbol{N}| \times |\boldsymbol{\Omega}|$ subproblems \eqref{STEPI_2}, corresponding to \eqref{eq:CADMM_storage}, with local variables $\boldsymbol{z}^{\boldsymbol{\mathrm{s}}}_{n,\omega}$, and feasible set $\boldsymbol{\Gamma}^{\boldsymbol{\mathrm{s}}}_{n,\omega}\left(\boldsymbol{\theta}^{\boldsymbol{\mathrm{s}}}_{n,\omega}\right)$, defined by \eqref{AGG_SMPC:storage_soc}--\eqref{AGG_SMPC:storage_dischar_bounds} and \eqref{AGG_SMPC:pos_vars}, with parameters $\boldsymbol{\theta}^{\boldsymbol{\mathrm{s}}}_{n,\omega}$.
    \item $|\boldsymbol{G}| \times |\boldsymbol{\Omega}|$ subproblems \eqref{STEPI_3}, corresponding to \eqref{eq:CADMM_thermal}, with local variables $\boldsymbol{\tilde{p}}^{\boldsymbol{\mathrm{th}}}_{g,\omega}$ and feasible set $\boldsymbol{\Gamma}^{\boldsymbol{\mathrm{th}}}_{g,\omega}(\boldsymbol{\theta}^{\boldsymbol{\mathrm{th}}}_{g,\omega})$ defined by \eqref{AGG_SMPC:thr_bounds}--\eqref{AGG_SMPC:thr_ramp2} and \eqref{AGG_SMPC:pos_vars}, with parameters $\boldsymbol{\theta}^{\boldsymbol{\mathrm{th}}}_{g,\omega}$.
\end{itemize}
Consistency among the local variables is enforced by the consensus constraints \eqref{eq:CADMM_consensus1}--\eqref{eq:CADMM_consensus4}, with associated dual variables
$\boldsymbol{\hat{\lambda}}^{\boldsymbol{\mathrm{s}}}_{n,\omega}$,
$\boldsymbol{\tilde{\lambda}}^{\boldsymbol{\mathrm{s}}}_{n,\omega}$,
$\boldsymbol{\hat{\lambda}}^{\boldsymbol{\mathrm{th}}}_{g,\omega}$,
and 
$\boldsymbol{\tilde{\lambda}}^{\boldsymbol{\mathrm{th}}}_{g,\omega}$.
Separately for each scenario, the subproblems \eqref{STEPI_2} and \eqref{STEPI_3} determine the dispatch of each storage and thermal unit, respectively, while \eqref{STEPI_1} enforces the VPP energy balance.
The derived $(|\boldsymbol{N}| + |\boldsymbol{G}|)\times|\boldsymbol{\Omega}| + 1$ subproblems are solved \textbf{in parallel} via consensus ADMM.

Let $\rho$ be the ADMM step size. At time $t$, our \textbf{temporally aggregated and distributed stochastic MPC scheme} executes the following steps over $i \!\in\! \boldsymbol{I}$ iterations, with $I \!\coloneqq\! |\boldsymbol{I}|$.

\noindent\textbf{Step I}. Local primal variable update:
\begin{align}
\label{STEPI_1}
    \boldsymbol{z}^{\boldsymbol{\mathrm{b}}^{i+1}} \!\!\! \coloneqq & \argmin_{\boldsymbol{z}^{\boldsymbol{\mathrm{b}}} \in \boldsymbol{\Xi}} \Bigg\{ \! \sum_{\omega \in \boldsymbol{\Omega}} \phi_{\omega}  \sum_{r \in \boldsymbol{R}} K_r \left( - \max_{k \in \boldsymbol{K}_r}\left\{\pi_{\omega,k}\right\} \bar{e}_{\omega,r} \! \right) \nonumber\\
    & + \sum_{\omega \in \boldsymbol{\Omega}} \phi_{\omega} \sum_{r \in \boldsymbol{R}} K_r \, C^{\mathrm{ns}} \, \bar{e}^{\mathrm{ns}}_{\omega,r} \nonumber\\
    & + \!\! \sum_{n \in \boldsymbol{N}} \sum_{\omega \in \boldsymbol{\Omega}} \!\! \left( \!\! \left( \boldsymbol{\hat{\lambda}}^{\boldsymbol{\mathrm{s}}^i}_{n,\omega} \right)^{\!\!\top} \!\! \boldsymbol{\hat{p}}^{\boldsymbol{\mathrm{s}}}_{n,\omega} \!+\! \frac{\rho}{2} \! \left\|\boldsymbol{\hat{p}}^{\boldsymbol{\mathrm{s}}}_{n,\omega} \!-\! \boldsymbol{\bar{p}}^{\boldsymbol{\mathrm{s}}^i}_{n,\omega}\right\|_2^2 \right) \nonumber\\
    & + \!\! \sum_{g \in \boldsymbol{G}} \sum_{\omega \in \boldsymbol{\Omega}} \!\! \left( \!\! \left( \boldsymbol{\hat{\lambda}}^{\boldsymbol{\mathrm{th}}^i}_{g,\omega} \right)^{\!\!\top} \!\! \boldsymbol{\hat{p}}^{\boldsymbol{\mathrm{th}}}_{g,\omega} \!+\! \frac{\rho}{2} \! \left\|\boldsymbol{\hat{p}}^{\boldsymbol{\mathrm{th}}}_{g,\omega} \!-\! \boldsymbol{\bar{p}}^{\boldsymbol{\mathrm{th}}^i}_{g,\omega}\right\|_2^2 \!\right) \!\!\! \Bigg\},
\end{align}
\begin{align}
\label{STEPI_2}
    \boldsymbol{z}^{\boldsymbol{\mathrm{s}}^{i+1}}_{n,\omega} \coloneqq & \argmin_{\boldsymbol{z}^{\boldsymbol{\mathrm{s}}}_{n,\omega} \in \boldsymbol{\Gamma}^{\boldsymbol{\mathrm{s}}}_{n,\omega}} 
    \Bigg\{ \phi_{\omega} \sum_{r \in \boldsymbol{R}} K_r \left( C^{\mathrm{c}}_n \tilde{p}^\mathrm{c}_{n,\omega,r} + C^{\mathrm{d}}_n \tilde{p}^\mathrm{d}_{n,\omega,r} \right) \Delta \nonumber\\
    & + \phi_{\omega} \sum_{r \in \boldsymbol{R}} C^{\mathrm{ref}}_n \left(\bar{e}^{\mathrm{s}}_{n,\omega,r} - E^{\mathrm{ref}}_n\right)^2 + \left( \boldsymbol{\tilde{\lambda}}^{\boldsymbol{\mathrm{s}}^i}_{n,\omega} \right)^{\!\!\top} \! \boldsymbol{\tilde{p}}^{\boldsymbol{\mathrm{s}}}_{n,\omega} \nonumber\\
    & + \frac{\rho}{2} \left\|\boldsymbol{\tilde{p}}^{\boldsymbol{\mathrm{s}}}_{n,\omega} - \boldsymbol{\bar{p}}^{\boldsymbol{\mathrm{s}}^i}_{n,\omega}\right\|_2^2 \Bigg\}, \; \forall n, \forall \omega,
\end{align}
\begin{align}
\label{STEPI_3}
    \boldsymbol{\tilde{p}}^{\boldsymbol{\mathrm{th}}^{i+1}}_{g,\omega} \coloneqq & \argmin_{\boldsymbol{\tilde{p}}^{\boldsymbol{\mathrm{th}}}_{g,\omega} \in  \boldsymbol{\Gamma}^{\boldsymbol{\mathrm{th}}}_{g,\omega}} \Bigg\{
    \phi_{\omega} \sum_{r \in \boldsymbol{R}} C^{\mathrm{th}}_g \tilde{p}^\mathrm{th}_{g,\omega,r} K_r \Delta + \left( \boldsymbol{\tilde{\lambda}}^{\boldsymbol{\mathrm{th}}^i}_{g,\omega} \right)^{\!\!\top} \! \boldsymbol{\tilde{p}}^{\boldsymbol{\mathrm{th}}}_{g,\omega} \nonumber\\
    & + \frac{\rho}{2} \left\|\boldsymbol{\tilde{p}}^{\boldsymbol{\mathrm{th}}}_{g,\omega} - \boldsymbol{\bar{p}}^{\boldsymbol{\mathrm{th}}^i}_{g,\omega}\right\|_2^2 \Bigg\}, \; \forall g, \forall \omega.
\end{align}

\noindent\textbf{Step II}. Global primal variable update:
\begin{align}
    \boldsymbol{\bar{p}}^{\boldsymbol{\mathrm{s}}^{i+1}}_{n,\omega} & \coloneqq \frac{1}{2} \left(\boldsymbol{\hat{p}}^{\boldsymbol{\mathrm{s}}^{i+1}}_{n,\omega} + \boldsymbol{\tilde{p}}^{\boldsymbol{\mathrm{s}}^{i+1}}_{n,\omega}\right), \; \forall n, \forall \omega, \label{STEPII_1}\\
    \boldsymbol{\bar{p}}^{\boldsymbol{\mathrm{th}}^{i+1}}_{g,\omega} & \coloneqq \frac{1}{2} \left(\boldsymbol{\hat{p}}^{\boldsymbol{\mathrm{th}}^{i+1}}_{g,\omega} + \boldsymbol{\tilde{p}}^{\boldsymbol{\mathrm{th}}^{i+1}}_{g,\omega}\right), \; \forall g, \forall \omega. \label{STEPII_2}
\end{align}

\noindent\textbf{Step III}. Dual variable update:
\begin{align}
    \boldsymbol{\hat{\lambda}}^{\boldsymbol{\mathrm{s}}^{i+1}}_{n,\omega} & \coloneqq \boldsymbol{\hat{\lambda}}^{\boldsymbol{\mathrm{s}}^{i}}_{n,\omega} + \rho \left(\boldsymbol{\hat{p}}^{\boldsymbol{\mathrm{s}}^{i+1}}_{n,\omega} - \boldsymbol{\bar{p}}^{\boldsymbol{\mathrm{s}}^{i+1}}_{n,\omega}\right), \; \forall n, \forall \omega, \label{STEPIII_1}\\
    \boldsymbol{\tilde{\lambda}}^{\boldsymbol{\mathrm{s}}^{i+1}}_{n,\omega} & \coloneqq \boldsymbol{\tilde{\lambda}}^{\boldsymbol{\mathrm{s}}^{i}}_{n,\omega} + \rho \left(\boldsymbol{\tilde{p}}^{\boldsymbol{\mathrm{s}}^{i+1}}_{n,\omega} - \boldsymbol{\bar{p}}^{\boldsymbol{\mathrm{s}}^{i+1}}_{n,\omega}\right), \; \forall n, \forall \omega, \label{STEPIII_2}\\
    \boldsymbol{\hat{\lambda}}^{\boldsymbol{\mathrm{th}}^{i+1}}_{g,\omega} & \coloneqq \boldsymbol{\hat{\lambda}}^{\boldsymbol{\mathrm{th}}^{i}}_{g,\omega} + \rho \left(\boldsymbol{\hat{p}}^{\boldsymbol{\mathrm{th}}^{i+1}}_{g,\omega} - \boldsymbol{\bar{p}}^{\boldsymbol{\mathrm{th}}^{i+1}}_{g,\omega}\right), \; \forall g, \forall \omega, \label{STEPIII_3}\\
    \boldsymbol{\tilde{\lambda}}^{\boldsymbol{\mathrm{th}}^{i+1}}_{g,\omega} & \coloneqq \boldsymbol{\tilde{\lambda}}^{\boldsymbol{\mathrm{th}}^{i}}_{g,\omega} + \rho \left(\boldsymbol{\tilde{p}}^{\boldsymbol{\mathrm{th}}^{i+1}}_{g,\omega} - \boldsymbol{\bar{p}}^{\boldsymbol{\mathrm{th}}^{i+1}}_{g,\omega}\right), \; \forall g, \forall \omega. \label{STEPIII_4}
\end{align}
To enhance convergence, we adopt the residual-based update rule for $\rho$, initialized at $\rho^{0}$, as proposed in \cite{Boyd}.

In the following, we adopt a synchronous implementation of the consensus ADMM procedure, whereby all distributed subproblems associated with iteration $i$ are solved in parallel before proceeding to iteration $i+1$.
Nevertheless, we emphasize that the proposed methodology is not inherently tied to a synchronous implementation and can be readily extended to an asynchronous ADMM implementation \cite{wang2022asynchronous}.

\subsection{Derivation of the Objective Function Bounds}
\label{subsec:TheoreticalResult}

In Subsection~\ref{subsec:AggregatedDistributed}, the aggregated model \eqref{AGG_SMPC} is decomposed using ADMM.
Since \eqref{AGG_SMPC} is convex, ADMM is guaranteed to converge to the optimal objective function value of the aggregated model \eqref{AGG_SMPC} under the standard ADMM convergence assumptions \cite{Boyd}.
In the numerical implementation, sufficiently stringent primal and dual residual tolerances are adopted to obtain accurate numerical approximations of the optimal objective function value of \eqref{AGG_SMPC}.
However, \eqref{AGG_SMPC} is an approximation of the original full-scale model \eqref{FS_SMPC}.
To establish a formal performance guarantee, we bound the resulting approximation error as follows.

\begin{proposition}\label{prop:main_result}
Let $\boldsymbol{z}$ be a feasible solution to the full-scale model \eqref{FS_SMPC}.
Let $\boldsymbol{\bar{z}}$ be derived from $\boldsymbol{z}$ as

\vspace{-0.15cm}
\begin{minipage}{0.45\columnwidth}
\begin{equation}
\label{prop:main_result_eq1}
\bar{e}_{\omega,r} \coloneqq \sum_{k \in \boldsymbol{K}_r} \frac{e_{\omega,k}}{K_r},
\end{equation}
\end{minipage}%
\hfill
\begin{minipage}{0.45\columnwidth}
\begin{equation}
\label{prop:main_result_eq2}
\bar{e}^{\mathrm{ns}}_{\omega,r} \coloneqq \sum_{k \in \boldsymbol{K}_r} \frac{e^{\mathrm{ns}}_{\omega,k}}{K_r},
\end{equation}
\end{minipage}

\vspace{-0.15cm}
\begin{minipage}{0.45\columnwidth}
\begin{equation}
\label{prop:main_result_eq3}
\bar{e}^{\mathrm{s}}_{n,\omega,r} \!\!\coloneqq\! e^{\mathrm{s}}_{n,\omega,\min(\!\boldsymbol{K}_r\!)},
\end{equation}
\end{minipage}%
\hfill
\begin{minipage}{0.45\columnwidth}
\begin{equation}
\label{prop:main_result_eq4}
\bar{p}^{\mathrm{th}}_{g,\omega,r} \!\coloneqq\! \sum_{k \in \boldsymbol{K}_r} \!\! \frac{p^{\mathrm{th}}_{g,\omega,k}}{K_r},
\end{equation}
\end{minipage}

\vspace{-0.15cm}
\begin{minipage}{0.45\columnwidth}
\begin{equation}
\label{prop:main_result_eq5}
\bar{p}^{\mathrm{c}}_{n,\omega,r} \!\coloneqq\!\! \sum_{k \in \boldsymbol{K}_r} \!\! \frac{p^{\mathrm{c}}_{n,\omega,k}}{K_r},
\end{equation}
\end{minipage}%
\hfill
\begin{minipage}{0.45\columnwidth}
\begin{equation}
\label{prop:main_result_eq6}
\bar{p}^{\mathrm{d}}_{n,\omega,r} \!\coloneqq\!\! \sum_{k \in \boldsymbol{K}_r} \!\! \frac{p^{\mathrm{d}}_{n,\omega,k}}{K_r},
\end{equation}
\end{minipage}
\vspace{0.1cm}

$\forall n, \forall g, \forall \omega, \forall r$. Then, $\boldsymbol{\bar{z}}$ is a feasible solution to the temporally aggregated model \eqref{AGG_SMPC} and it holds that $\bar{F}\left(\boldsymbol{\bar{z}}\right) \leq F\left(\boldsymbol{z}\right)$.
\end{proposition}

The proof of Proposition~\ref{prop:main_result} is provided in the Appendix.

In words, Proposition~\ref{prop:main_result} asserts that every feasible solution of the full-scale model \eqref{FS_SMPC} corresponds to a feasible solution of the temporally aggregated model \eqref{AGG_SMPC} with an equal or lower objective function value.
Consequently, the controller proposed in Subsection~\ref{subsec:AggregatedDistributed} provides a guaranteed lower bound on the optimal objective function value of the original full-scale controller in Subsection~\ref{subsec:Full-Scale} at each MPC iteration.

Once a lower bound $F^{\mathrm{LB}}$ on the full-scale optimal objective function value $F^\star$ is obtained, as established in Proposition~\ref{prop:main_result},
an upper bound $F^{\mathrm{UB}}$ is computed by solving \eqref{FS_SMPC} with the first-stage decisions $\boldsymbol{u}_{0}$ fixed to those obtained from solving \eqref{AGG_SMPC}.
Under this restriction, \eqref{FS_SMPC} can be solved in parallel across scenarios.

We remark that the computation of the upper bound is supported by two key properties.
First, as described in Subsection~\ref{subsec:APosterioriTSA}, the temporally aggregated model always preserves $k=0$ as a singleton cluster,
ensuring that the first-stage decisions obtained from \eqref{AGG_SMPC} satisfy the corresponding constraints of the full-scale model \eqref{FS_SMPC}.
Second, in the considered dispatch formulation \eqref{FS_SMPC}, both the energy exported to the market and the non-supplied energy are constrained to be nonnegative, but neither variable is bounded from above.
Consequently, any surplus generation can be accommodated through market exports, while any generation deficit can be compensated through non-supplied energy.
Under these assumptions, the full-scale model satisfies the relatively complete recourse property,
thereby guaranteeing that fixing the first-stage decisions in \eqref{FS_SMPC} to those obtained from \eqref{AGG_SMPC} preserves feasibility.
This guarantee does not necessarily extend to alternative dispatch formulations in which market exports or non-supplied energy are bounded from above.

The resulting \textbf{performance-guaranteed temporally aggregated and distributed stochastic MPC scheme} executes Algorithm~\ref{alg:dis_sto_MPC} at each time step $t \in \boldsymbol{T}$.

\begin{algorithm}[t]
\caption{Performance-Guaranteed Distributed Stochastic Predictive Control with A Posteriori Time Series Aggregation}\label{alg:dis_sto_MPC}
\begin{algorithmic}[1]
\Require $\boldsymbol{\beta}$, $\left\{\!\boldsymbol{\theta}^{\boldsymbol{\mathrm{s}}}_{n,\omega}, \!\boldsymbol{\theta}^{\boldsymbol{\mathrm{th}}}_{g,\omega}\!\right\}_{n \in \boldsymbol{N}, g \in \boldsymbol{G}, \omega \in \boldsymbol{\Omega}}$, $\gamma$, $\rho^0$, $\epsilon^{\mathrm{thr}}$, $\zeta^0$, $J$, $I$.

\vspace{0.05cm}
\Ensure Bounds $F^{\mathrm{UB}^\star}$ and $F^{\mathrm{LB}^\star}$, and control actions
$\boldsymbol{\bar{u}}_0^{\boldsymbol{\mathrm{best}}}$.

\State $j \gets 0$, $\zeta^j \gets \zeta^0$, $F^{\mathrm{UB}^{j}} \gets + \infty$, $F^{\mathrm{LB}^{j}} \gets - \infty$;

\State \textit{Higher layer - Temporal aggregation}:

\vspace{0.05cm}
\While{$100 \left|\frac{F^{\mathrm{UB}^{j}} - F^{{\mathrm{LB}}^{j}}}{\max\left(\left|F^{{\mathrm{UB}}^{j}}\right|,1\right)}\right| > \epsilon^{\mathrm{thr}}$ and $j < J$}

\vspace{0.05cm}
\State{
\parbox[t]{\dimexpr\linewidth-\algorithmicindent}{Perform TSA as detailed in Subsection~\ref{subsec:APosterioriTSA}, using \\ $\zeta^{j}$ and the average marginal costs across scenarios \\ from the previous MPC iteration as features;}
}

\vspace{0.05cm}
\State \parbox[t]{\dimexpr\linewidth-\algorithmicindent}{$i \!\gets\! 0$, $\rho \!\gets\! \rho^0$,
$\Big\{
\boldsymbol{\bar{p}}^{\boldsymbol{\mathrm{s}}^{0}}_{n,\omega} \!\gets\! \mathbf{0}^{2R},
\boldsymbol{\bar{p}}^{\boldsymbol{\mathrm{th}}^{0}}_{g,\omega} \!\gets\! \mathbf{0}^{R},
\boldsymbol{\hat{\lambda}}^{\boldsymbol{\mathrm{s}}^{0}}_{n,\omega} \!\gets\! \mathbf{0}^{2R},$
$\boldsymbol{\tilde{\lambda}}^{\boldsymbol{\mathrm{s}}^{0}}_{n,\omega} \gets \mathbf{0}^{2R},
\boldsymbol{\hat{\lambda}}^{\boldsymbol{\mathrm{th}}^{0}}_{g,\omega} \gets \mathbf{0}^{R}, 
\boldsymbol{\tilde{\lambda}}^{\boldsymbol{\mathrm{th}}^{0}}_{g,\omega} \gets \mathbf{0}^{R},
\; \forall n, \forall g, \forall \omega \Big\}$;}

\vspace{0.05cm}
\State \textit{Lower layer - Consensus ADMM}:
\vspace{0.05cm}
\While{the residual-based terminal condition from \cite{Boyd} \newline\hspace*{1.25em} is not satisfied and $i < I$}

\vspace{0.05cm}
\State $\left\{ \boldsymbol{z}^{\boldsymbol{\mathrm{b}}^{i+1}}, \boldsymbol{z}^{\boldsymbol{\mathrm{s}}^{i+1}}_{n,\omega}, \boldsymbol{\tilde{p}}^{\boldsymbol{\mathrm{th}}^{i+1}}_{g,\omega} \right\} \gets$ Solve \eqref{STEPI_1}--\eqref{STEPI_3};

\vspace{0.05cm}
\State $\left\{ \boldsymbol{\bar{p}}^{\boldsymbol{\mathrm{s}}^{i+1}}_{n,\omega}, \boldsymbol{\bar{p}}^{\boldsymbol{\mathrm{th}}^{i+1}}_{g,\omega} \right\} \gets$ Solve \eqref{STEPII_1}--\eqref{STEPII_2};

\vspace{0.05cm}
\State $\left\{ \!
\boldsymbol{\hat{\lambda}}^{\boldsymbol{\mathrm{s}}^{i+1}}_{n,\omega}, 
\boldsymbol{\tilde{\lambda}}^{\boldsymbol{\mathrm{s}}^{i+1}}_{n,\omega}, 
\boldsymbol{\hat{\lambda}}^{\boldsymbol{\mathrm{th}}^{i+1}}_{g,\omega}, 
\boldsymbol{\tilde{\lambda}}^{\boldsymbol{\mathrm{th}}^{i+1}}_{g,\omega}
\!\right\} \!\gets $ Solve \eqref{STEPIII_1}--\eqref{STEPIII_4};

\vspace{0.05cm}
\State Residual-based adaptive update of $\rho$ from \cite{Boyd};

\State $i \gets i+1$;

\EndWhile

\State $\boldsymbol{\bar{u}}_0^\star \gets$ Solve \eqref{AGG_SMPC:non_ant}, and $\bar{F}^\star \gets$ Solve \eqref{AGG_SMPC:obj};

\State \parbox[t]{\dimexpr\linewidth-\algorithmicindent}{${F^\mathrm{PRJ}} \gets$ Solve the full-scale stochastic model \eqref{FS_SMPC} with $\boldsymbol{u}_0 = \boldsymbol{\bar{u}}_0^\star$;
\Comment{In parallel $\forall \omega$}}

\State ${F^\mathrm{LB}}^{j+1} \gets \bar{F}^\star$;

\If{$F^{\mathrm{PRJ}} < F^{\mathrm{UB},j}$}
    \State $\boldsymbol{\bar{u}}_0^{\boldsymbol{\mathrm{best}}} \leftarrow \boldsymbol{\bar{u}}_0^\star$;
\EndIf

\State ${F^\mathrm{UB}}^{j+1} \gets \min\big({F^\mathrm{UB}}^{j}, F^\mathrm{PRJ}\big)$;

\State $\zeta^{j+1} \gets \gamma \, \zeta^j$;
\State $j \gets j+1$;

\EndWhile

\State\textbf{return} ${{F^\mathrm{UB}}}^\star \gets {F^\mathrm{UB}}^j$, ${{F^\mathrm{LB}}}^\star \gets {F^\mathrm{LB}}^j$ and
$\boldsymbol{\bar{u}}_0^{\boldsymbol{\mathrm{best}}}$;

\end{algorithmic}
\end{algorithm}

Algorithm~\ref{alg:dis_sto_MPC} is structured into a higher and a lower layer, executing $j \in \boldsymbol{J}$ and $i \in \boldsymbol{I}$ iterations, respectively.
At time $t$, the \textbf{higher layer} receives the dual variable values computed at time $t-1$,
specifically the marginal cost estimates,
which serve as features for the sliding window clustering technique detailed in Subsection~\ref{subsec:APosterioriTSA}.
The identified clusters are then used to construct the temporally aggregated model \eqref{AGG_SMPC},
which is passed to the \textbf{lower layer}.
Here, the problem is decomposed into
$(|\boldsymbol{N}| + |\boldsymbol{G}|)\times|\boldsymbol{\Omega}| + 1$ 
temporally aggregated subproblems, which are solved in parallel across assets and scenarios via consensus ADMM.
Upon convergence of ADMM, determined according to the residual-based terminal condition of \cite{Boyd},
the solution of the aggregated model is employed to compute $F^{\mathrm{UB}}$ and $F^{\mathrm{LB}}$.
If the relative difference between these bounds falls below the threshold $\epsilon^{\mathrm{thr}}$, the algorithm terminates; otherwise $\zeta$, initialized to $\zeta^0$, is reduced by a positive factor $\gamma < 1$, and the algorithm iterates again.
Whenever the objective function value $F^{\mathrm{PRJ}}$ obtained by solving the full-scale model \eqref{FS_SMPC} with fixed first-stage decisions improves the current upper bound,
the corresponding first-stage control actions $\boldsymbol{\bar{u}}_0^\star$ are stored.
Upon termination, Algorithm~\ref{alg:dis_sto_MPC} returns the stored control actions
$\boldsymbol{\bar{u}}_0^{\boldsymbol{\mathrm{best}}}$,
associated with the tightest upper bound obtained throughout the iterative procedure,
which do not necessarily coincide with those computed at the final iteration of the algorithm.
Notably, the returned control actions are \textbf{feasible} for the full-scale model \eqref{FS_SMPC},
and their quality relative to an optimal solution of \eqref{FS_SMPC} can be assessed, in terms of objective function value, through the reported optimality gap
(i.e., the relative difference between the upper and lower bounds).

As illustrated in Fig.~\ref{fig:distributed_MPC_TSA}, Algorithm~\ref{alg:dis_sto_MPC} provides a formal performance guarantee by allowing the decision-maker to evaluate the achieved optimality gap at each iteration.

\section{Simulation Results and Discussion}
\label{sec:Results}

This section presents the case study (Subsection~\ref{sub:Results_CaseStudy}) and discusses the numerical results (Subsections~\ref{sub:Results_MarginalCost} and~\ref{sub:Results_ControllerPerformance}).

\subsection{Case Study Description}
\label{sub:Results_CaseStudy}

We simulate the VPP energy dispatch over one year
using a prediction horizon of $24$~h and a sampling time of $\Delta = 5$~min, i.e., $|\boldsymbol{K}| = 288$.
To ensure transparency and reproducibility, the VPP parameters are sampled from prescribed distributions.

For each storage unit, the minimum state of charge is set to $0$~MWh,
while the energy capacity is sampled from a uniform distribution over $[1,5]$~MWh.
The maximum charging/discharging power limits are drawn uniformly from $[0.1~\mathrm{MW}, \overline{E}^{\mathrm{s}}_n/ (1~\mathrm{h})], \forall n$.
The initial state is set to $0$~MWh,
and charging and discharging efficiencies are sampled uniformly from $[0.7,0.9]$ and $[1/0.9,1/0.7]$, respectively.
The reference states are $\overline{E}^{\mathrm{s}}_n/2, \forall n$.

For each thermal unit, $\underline{P}^\mathrm{th}_g = 0$~MW,
the capacity is drawn uniformly from $[0,2]$~MW,
the ramp limit is $\overline{P}^{\mathrm{th}}_g/6$,
and the parameters in \eqref{model:thr_emission} are
{\small $\alpha_g \!=\! \frac{1}{2} \left(\overline{P}^{\mathrm{th}}_g\right)^{-2}$},
{\small $\beta_g \!=\! \left(2 \, \overline{P}^{\mathrm{th}}_g\right)^{-1}$}, and
{\small $L^{\mathrm{CO}_2}_g \!=\! 0.9$}.
These parameter choices ensure that the coefficients defining the emission constraints are explicitly determined as a function of the nominal capacity of each thermal generator.
Thermal and storage operation costs are drawn uniformly from $[30,60]$ and $[5,15]$~\EUR/MWh, respectively.
The cost of non-supplied energy is $5000$~\EUR/MWh,
and deviations from storage reference states incur a cost of $1$~\EUR/(MWh)$^2$.

\begin{figure}[t]
\centering
\includegraphics[scale=.6]{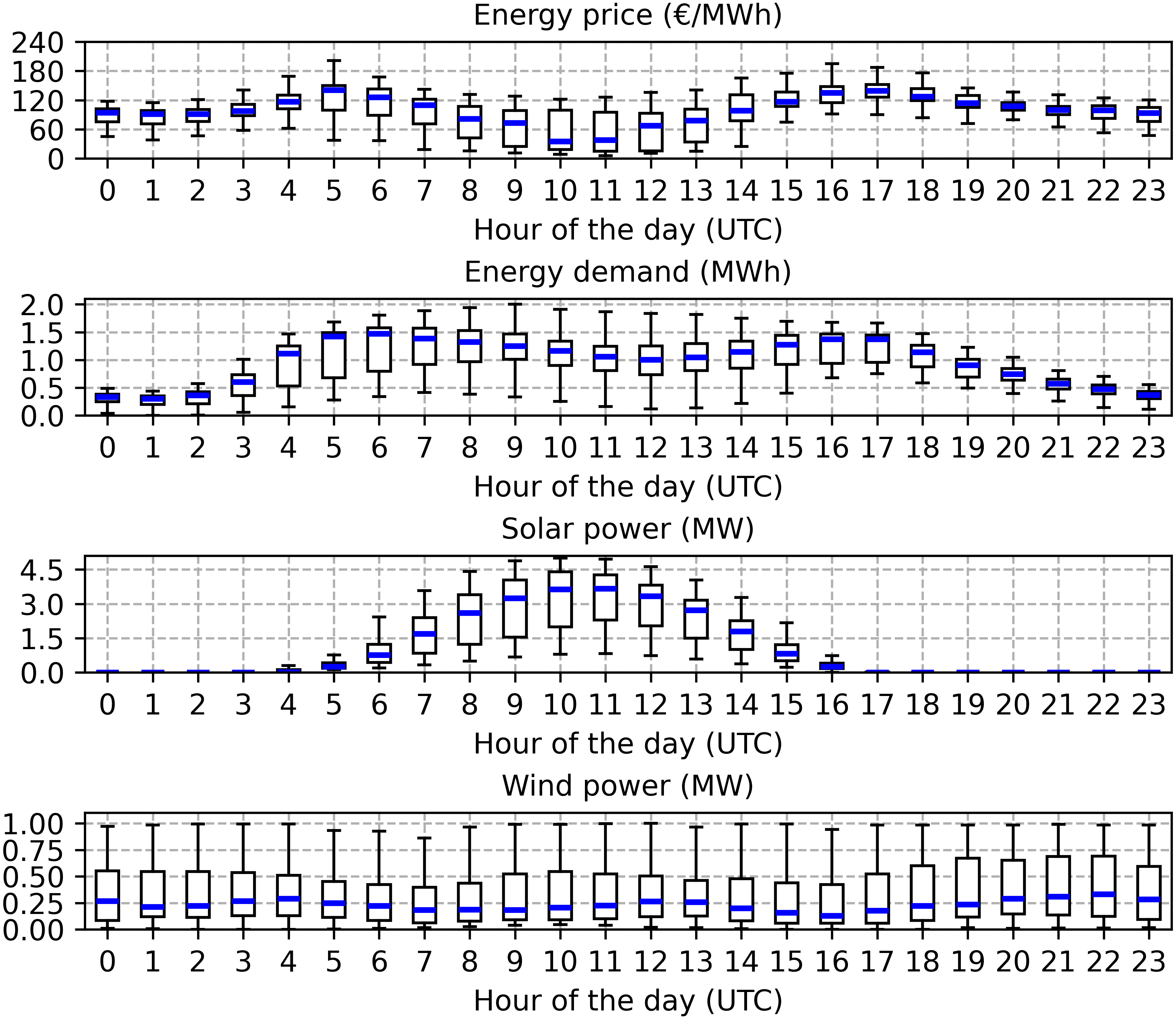}
\caption{Boxplots of the hourly uncertainty realization values. The box represents the interquartile range, the blue line indicates the median, and the whiskers extend to the minimum and maximum observations.}
\label{fig:input_data_boxplot}
\end{figure}

\begin{figure}[t]
\centering
\includegraphics[scale=.59  ]{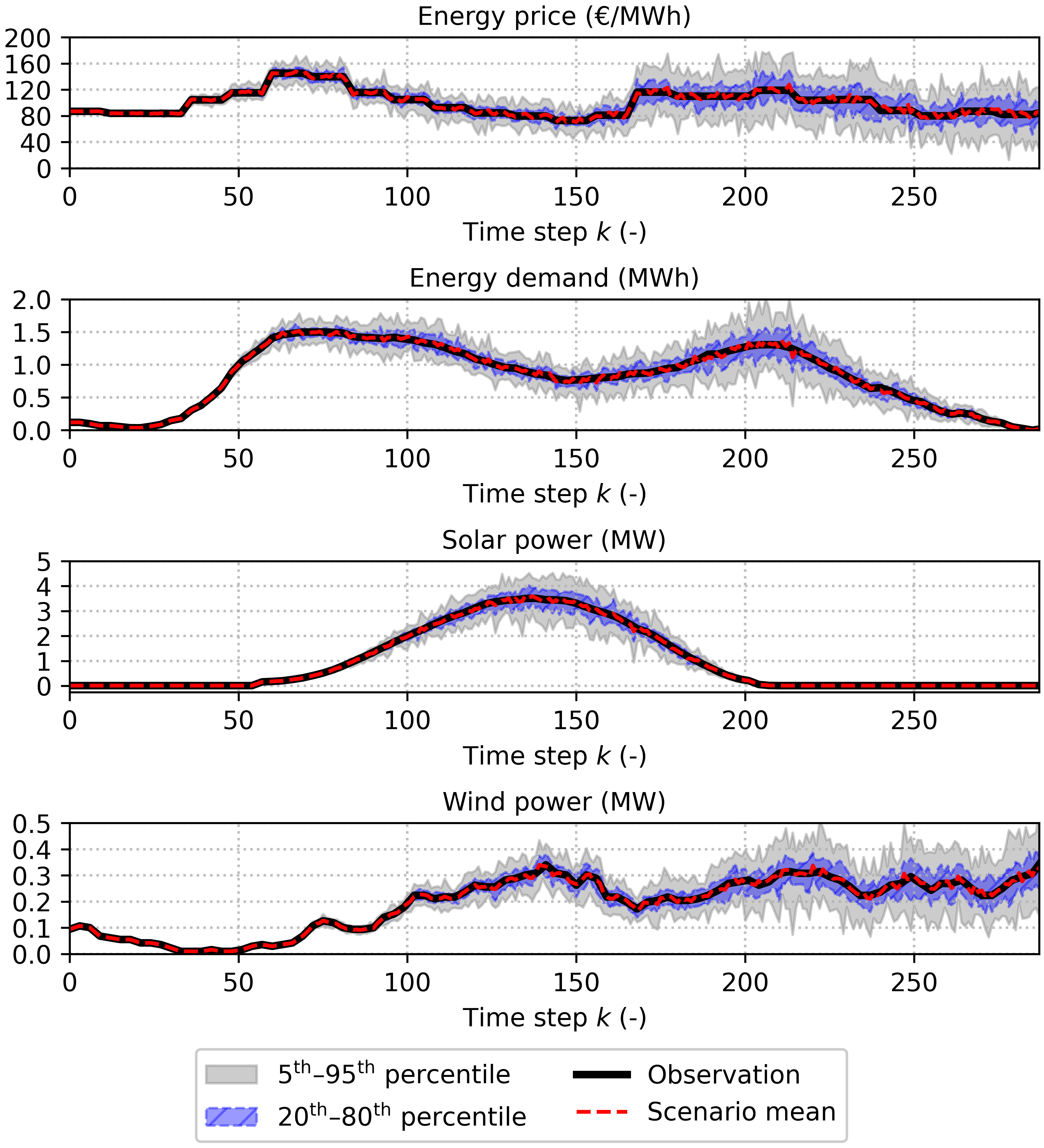}%
\caption{
Example of the generated scenario distributions over the MPC prediction horizon for $|\boldsymbol{\Omega}| = 50$.
The plots depict the baseline observation, the scenario mean, and the $5^{\mathrm{th}}$–$95^{\mathrm{th}}$ and $20^{\mathrm{th}}$–$80^{\mathrm{th}}$ percentile intervals characterizing the scenario distributions for each uncertainty source.
}
\label{fig:scenarios_example}
\end{figure}

The uncertainty realizations are represented by scaled day-ahead energy prices, energy demand, and wind and solar power generation in Austria for 2024,
obtained from the ENTSO-E Transparency Platform~\cite{hirth2018entso} and illustrated in Fig.~\ref{fig:input_data_boxplot}.
For each uncertainty source, equiprobable scenarios are generated by perturbing the baseline observations $y_k$
with multiplicative Laplace-distributed noise according to
\begin{equation}
    \hat{y}_{\omega,k} = \max\left(0,\, y_k (1 + \vartheta_\omega \, \frac{k}{K - 1} + \sigma_{\omega,k})\right), \; \forall \omega, \forall k.
\end{equation}
Here, $\hat{y}_{\omega,k}$ represents the forecast value in scenario $\omega$ at the MPC prediction step $k$, and $\sigma_{\omega,k} \sim \mathcal{L}(0,b_k)$ denotes Laplace-distributed noise with a scale parameter that increases linearly according to
$b_k = 0.25\,k/(K-1)$,
thereby reflecting a typical progressive increase in forecast uncertainty over the MPC prediction horizon.
For each scenario, $\vartheta_\omega$ is sampled uniformly from $[-0.05,\,0.05]$, introducing a scenario-specific bias.
An example of the resulting scenario distributions over the MPC prediction horizon is shown in Fig.~\ref{fig:scenarios_example}.

We emphasize that the adopted scenario generation approach is intended to ensure the transparency and reproducibility of the reported results while enabling the evaluation of the proposed methodology under realistic uncertainty conditions. The development of novel forecasting models for the considered uncertainty sources is beyond the scope of this work and is instead addressed in dedicated studies \cite{hong2020energy}.

If not otherwise specified, in Algorithm~\ref{alg:dis_sto_MPC} we set
$\gamma = 0.6$, $\rho^0 = 4$, $\epsilon^{\mathrm{thr}} = 1\%$, $\zeta^0 = 2$, $J = 10$, and $I = 500$.
Moreover, Algorithm~\ref{alg:dis_sto_MPC} is implemented using Python's multiprocessing module \cite{multiprocessing},
thereby enabling the fully parallel execution of both the distributed subproblems arising from the application of consensus ADMM to the temporally aggregated model \eqref{AGG_SMPC},
as detailed in Subsection \ref{subsec:AggregatedDistributed},
and the subproblems obtained by fixing the first-stage decisions in the full-scale model \eqref{FS_SMPC} to those computed from the temporally aggregated model \eqref{AGG_SMPC},
i.e., the algorithmic step yielding $F^{\mathrm{PRJ}}$ in Algorithm~\ref{alg:dis_sto_MPC}.
Each subproblem is assigned to an independent process with separate memory allocation.

The simulations are performed on an Intel i7 CPU with 32GB of RAM using Gurobi 12.0.1.

\subsection{Clustering Based on the Marginal Cost}
\label{sub:Results_MarginalCost}

As demonstrated in \cite{wogrin2023time}, exact temporal aggregation, i.e., zero error between the aggregated and full-scale models, is achieved when TSA yields an aggregated model that preserves the active constraints of the full-scale model.
In \eqref{FS_SMPC}, the marginal costs serve as an effective indicator of constraint activation over time:
variations in these dual variables reflect changes in the marginal generating units and thus implicitly encode the activation of dispatch constraints.
To illustrate the potential benefits of using marginal costs as clustering features for TSA, we consider the following examples.

We first examine a deterministic dispatch model ($|\boldsymbol{\Omega}| = 1$),
comprising $10$ thermal units.
Market participation is neglected, and each thermal unit is assigned an operating cost of $50$~\euro/MWh.
As shown in Fig.~\ref{fig:thermal_only_QCP_features_comparison},
applying the sliding window clustering technique of Subsection~\ref{subsec:APosterioriTSA} with similarity threshold $\zeta = 0$ and marginal costs as features,
yields distinct sets of clusters depending on the emission limit parameters $L^{\mathrm{CO}_2}_g$.
Specifically, $5$ to $7$ clusters are identified (including the singleton clusters corresponding to the first and last time steps in $\boldsymbol{K}$),
\textbf{enabling exact aggregation},
while reducing the temporal dimension by two orders of magnitude relative to the original $288$ time steps.
More importantly, Fig.~\ref{fig:thermal_only_QCP_features_comparison} shows that the clusters yielding exact TSA vary with the structural properties of the dispatch model (here, the emission limit parameters), even when the input time series (scenarios) are identical.
This is captured when performing TSA based on the marginal costs, whereas conventional a priori TSA methods, which rely solely on input time series analysis, generally fail.

\begin{figure}[t]
\centering
\includegraphics[scale=.59]{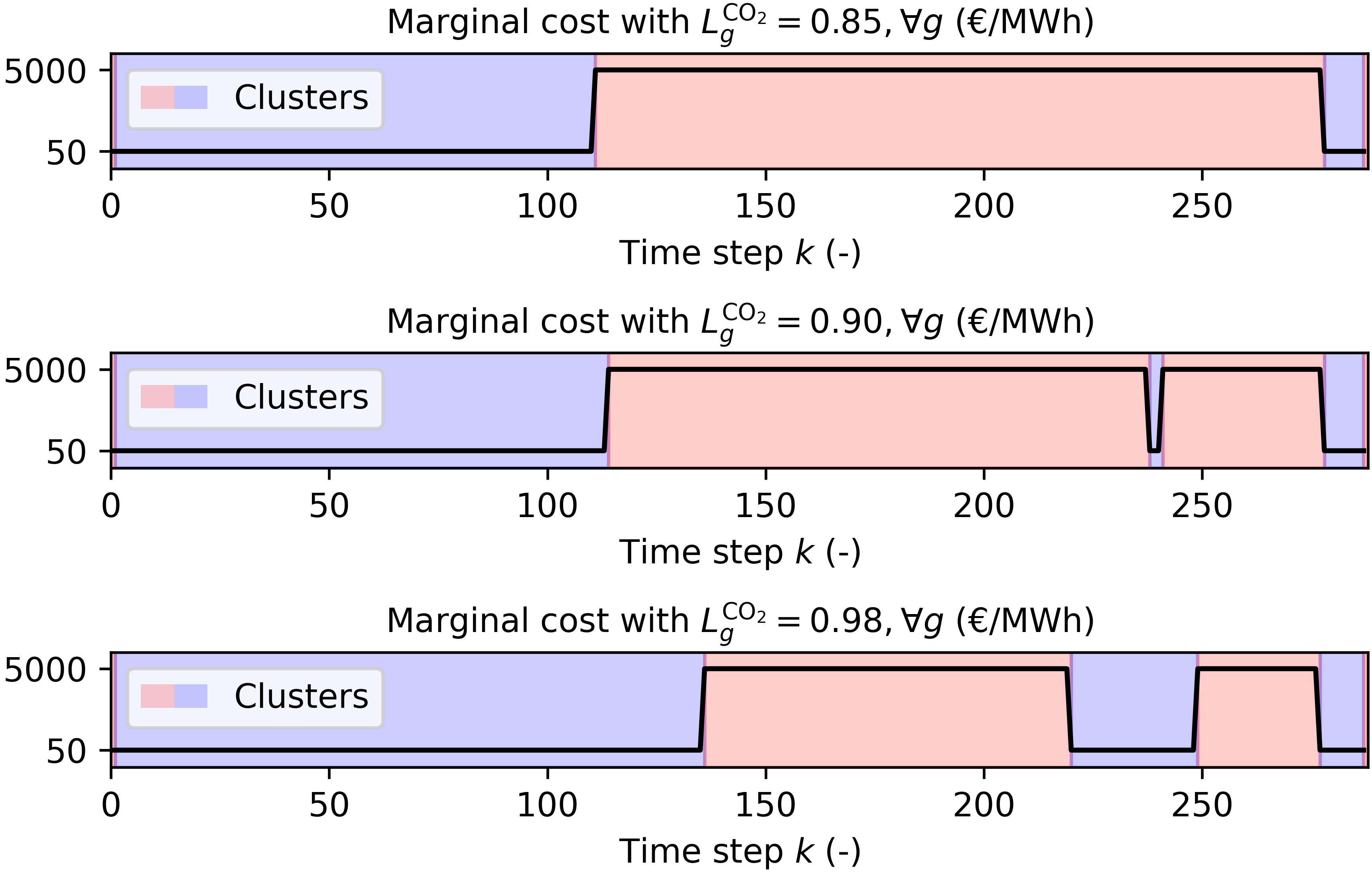}
\caption{Clusters obtained using marginal costs as features for TSA in the dispatch of thermal and solar power units under varying emission limits.}
\label{fig:thermal_only_QCP_features_comparison}
\end{figure}

For the same deterministic dispatch model comprising $10$ thermal units, Fig.~\ref{fig:constraints_act} further illustrates how marginal-cost-based clustering enables capturing the activation of the dispatch constraints.
The results correspond to the case $L^{\mathrm{CO}_2}_g = 0.98$ for all $g$,
i.e., the last subplot in Fig.~\ref{fig:thermal_only_QCP_features_comparison}.
The figure reports the normalized thermal emissions averaged across all generators, 
corresponding to the average value of the left-hand-side term in \eqref{model:thr_emission}, 
together with the occurrence of non-supplied energy and the associated marginal costs.
When the thermal units operate below their emission limits $L^{\mathrm{CO}_2}_g$, the emission constraints remain inactive and the marginal cost equals the thermal generation cost ($50$~\euro/MWh).
During periods of high demand, however, thermal output increases until the emission limits become binding.
Once these constraints are active, additional demand can only be satisfied through non-supplied energy, causing the marginal cost to increase from $50$~\euro/MWh to the non-supplied energy penalty of $5000$~\euro/MWh.
The marginal cost therefore provides a direct signature of changes in the active constraint set of the dispatch problem.

\begin{figure}[t]
\centering
\includegraphics[scale=.59]{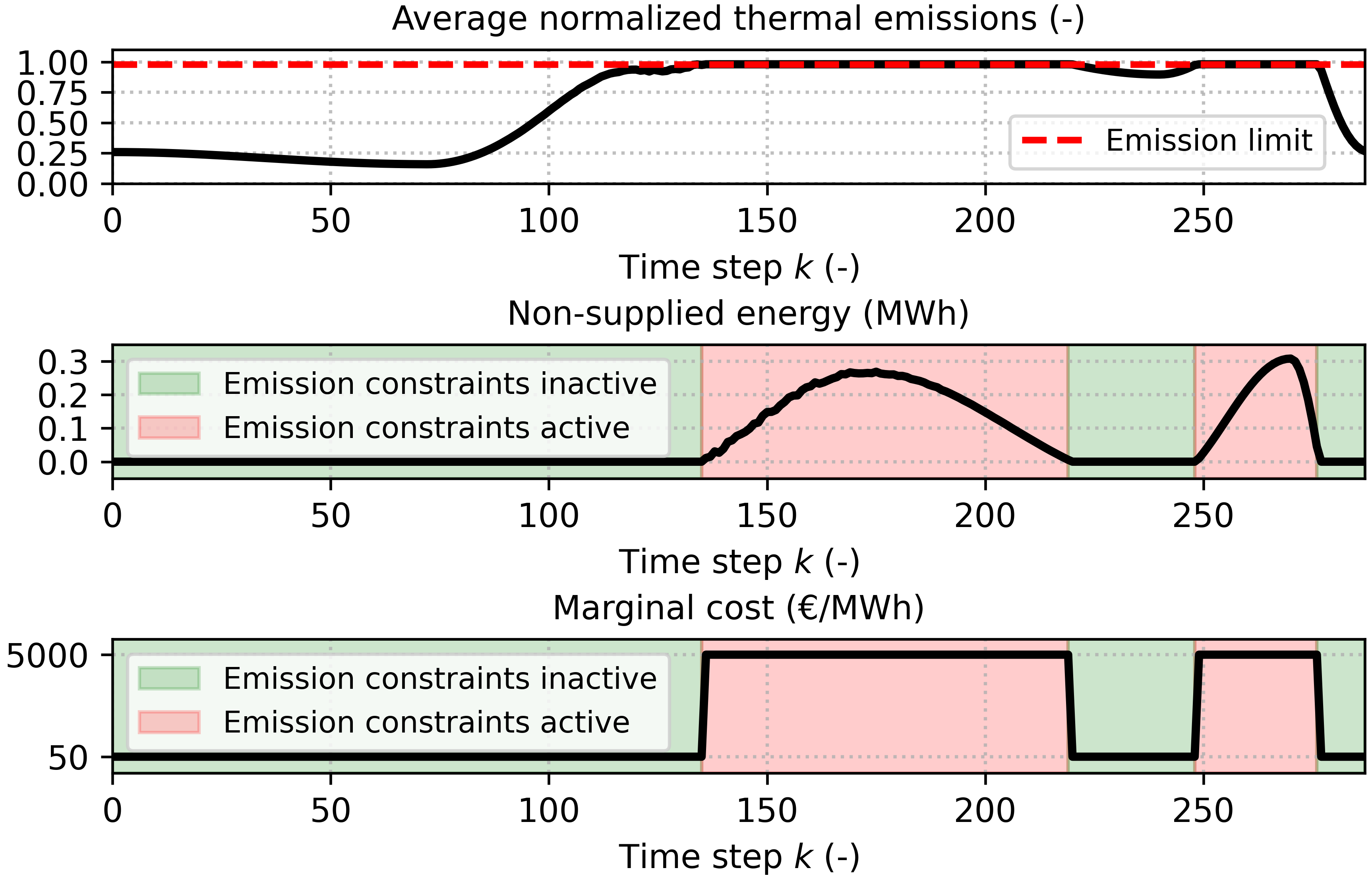}
\caption{Example of the relationship between marginal costs, the occurrence of non-supplied energy, and the activation of the emission constraints in the dispatch problem.
}
\label{fig:constraints_act}
\end{figure}

\begin{figure}[t]
\centering
\includegraphics[scale=.59]{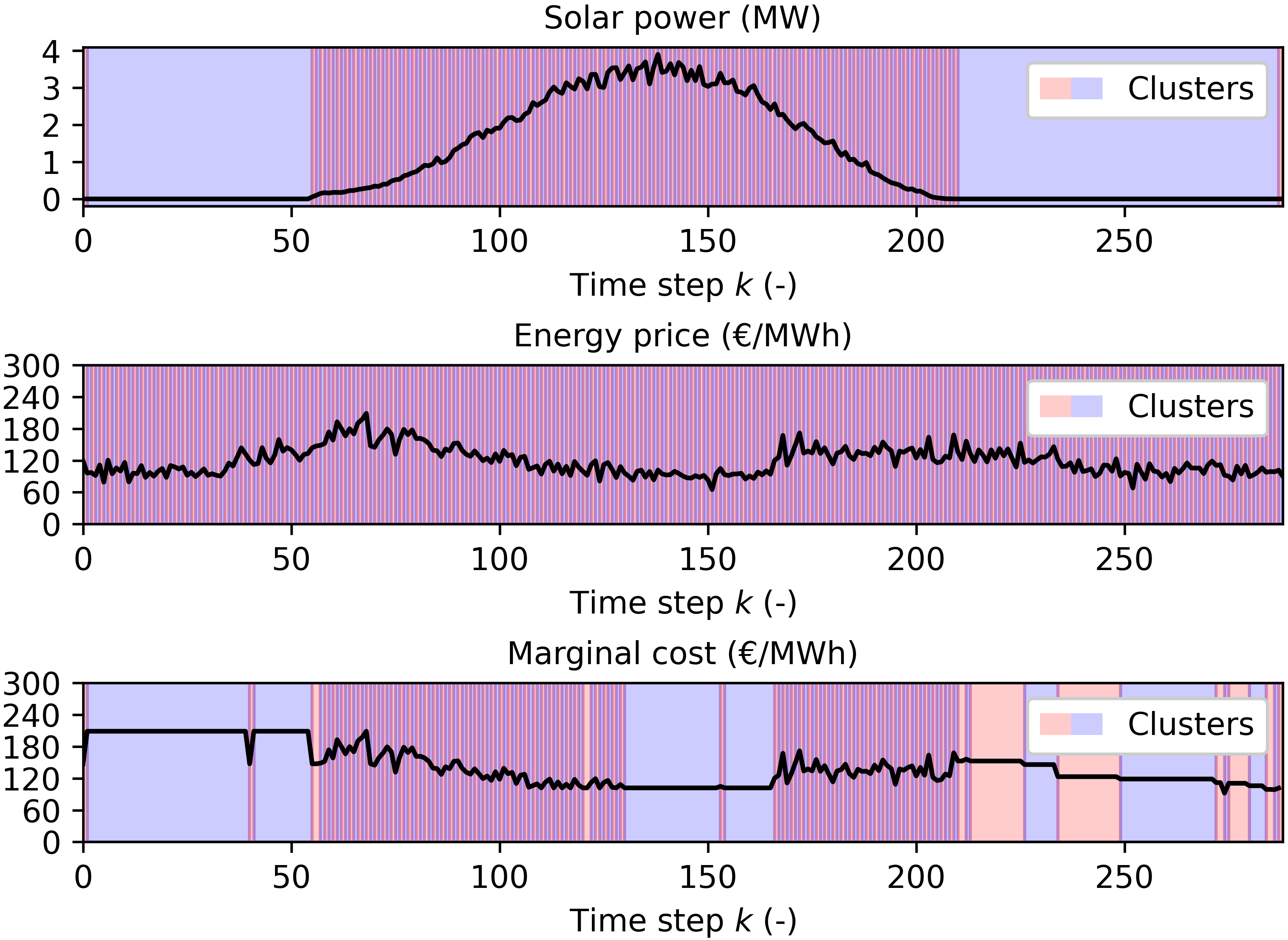}
\caption{Clusters obtained using different features (solar power, prices, or marginal costs) for TSA in the dispatch of storage and solar power units.}
\label{fig:bess_only_features_comparison}
\end{figure}

Fig.~\ref{fig:bess_only_features_comparison} reports the clusters obtained with our TSA method under different feature selections for a dispatch model with $10$ storage units.
The optimal objective function value of the full-scale dispatch model is $-3745.9$~\euro,
whereas the temporally aggregated models obtained using solar generation, energy prices, and marginal costs as clustering features yield optimal objective function values of $-4852.7$~\euro, $-3745.9$~\euro, and $-3745.9$~\euro, respectively.
Using solar generation as the clustering feature results in only $159$ clusters ($45\%$ reduction in the temporal dimension), but incurs a $30\%$ aggregation error.
Conversely, clustering based on the energy prices enables achieving the same optimal objective function value of the full-scale model, but provides no reduction in the temporal dimension.
Notably, using marginal costs as clustering features yields exact aggregation with only $137$ clusters, corresponding to a $52\%$ reduction in the temporal dimension,
thereby providing the most favorable trade-off between aggregation accuracy and temporal dimensionality reduction among the considered clustering features.
Compared to the thermal-only case in Fig.~\ref{fig:thermal_only_QCP_features_comparison},
a larger number of clusters is required,
reflecting the complexity induced by the intertemporal storage constraints.

\subsection{Performance Evaluation of the Proposed Controller}
\label{sub:Results_ControllerPerformance}

\begin{figure}[t]
    \centering
    \includegraphics[scale=.59]{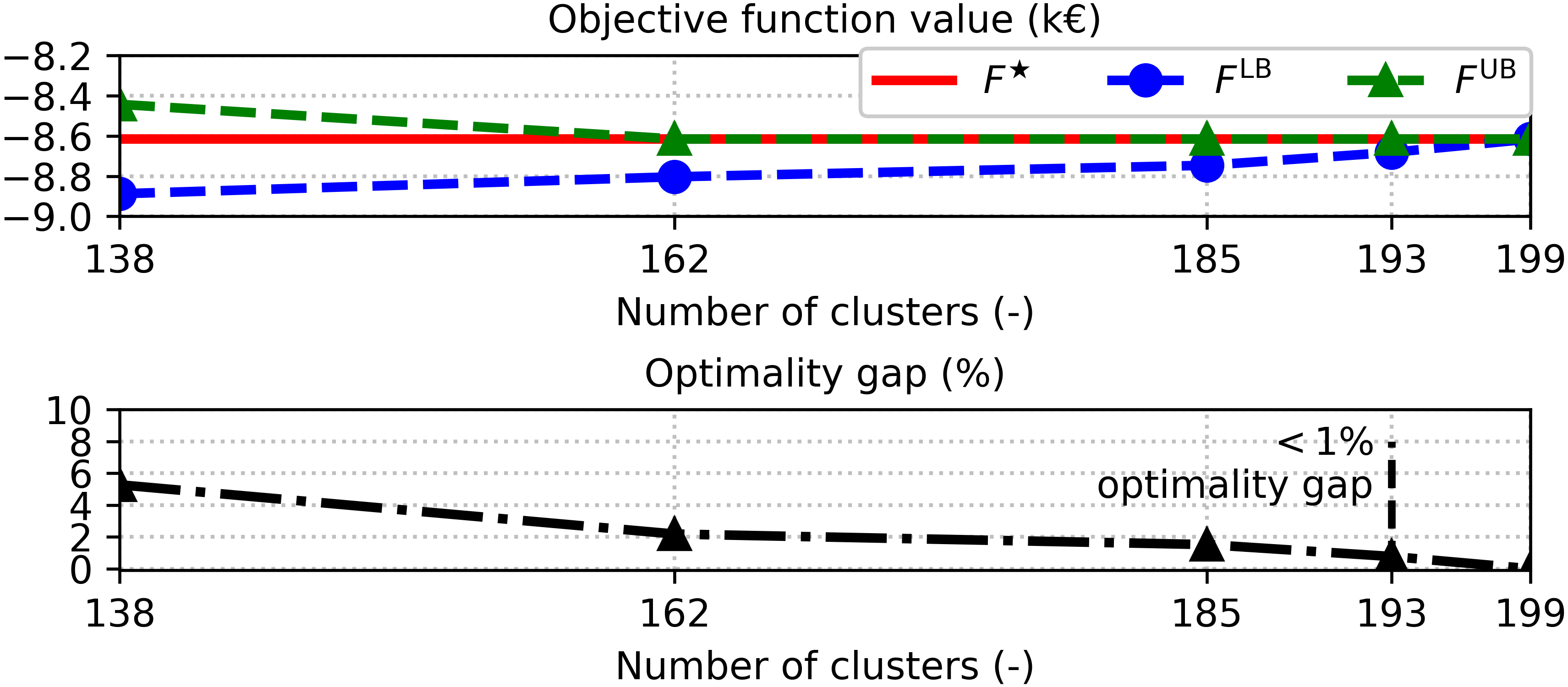}
    \caption{Example of objective function bounds computed using Algorithm~\ref{alg:dis_sto_MPC}.}
    \label{fig:bounds_example}
\end{figure}

\begin{figure}[t]
\centering
\includegraphics[scale=.59]{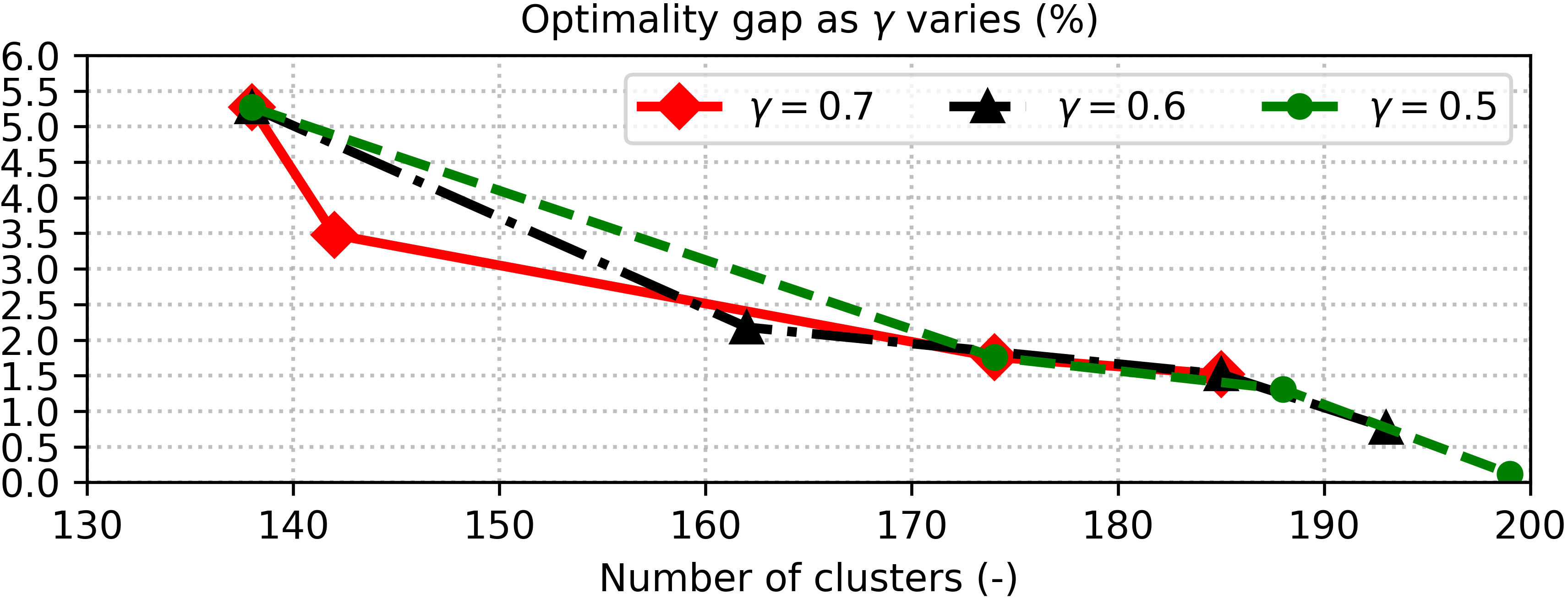}
\caption{Performance of Algorithm~\ref{alg:dis_sto_MPC} under varying values of $\gamma$.
}
\label{fig:gamma_sensitivity}
\end{figure}

\begin{figure}[t]
\centering
\includegraphics[scale=.59]{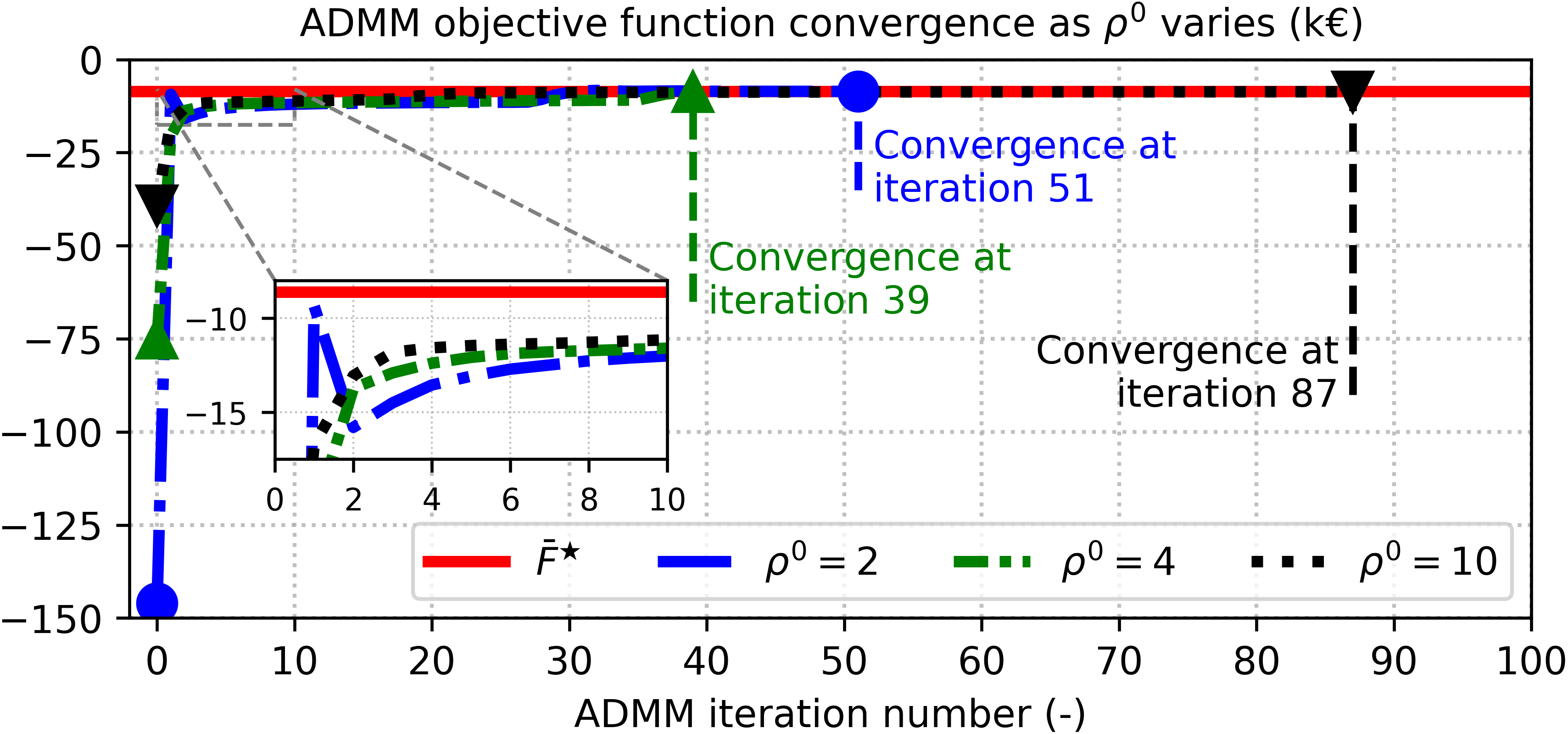}
\caption{Example of ADMM convergence within Algorithm~\ref{alg:dis_sto_MPC} as $\rho^0$ varies.}
\label{fig:ADMM_convergence}
\end{figure}

An example of the objective function bounds computed by Algorithm~\ref{alg:dis_sto_MPC} is shown in Fig.~\ref{fig:bounds_example},
for a dispatch problem with $|\boldsymbol{N}| =$ $|\boldsymbol{G}| =$ $|\boldsymbol{\Omega}| = 50$.
The similarity threshold $\zeta$ is initialized at $2$ and iteratively refined using $\gamma = 0.6$.
Therefore, over the $5$ iterations reported in the figure, the algorithm explores the following similarity threshold values: $\left\{2.0, 1.2, 0.7, 0.4, 0.3\right\}$.
The full-scale optimal objective function value is reported solely for reference, as it is not required by the algorithm.
The bounds are iteratively tightened, 
achieving an optimality gap of $\sim5\%$ with $138$ clusters,
and converging to a gap below $1\%$ (i.e., over $99\%$ accuracy relative to the full-scale controller) with $193$ clusters ($\sim 33\%$ reduction in the temporal dimension).

The performance of Algorithm~\ref{alg:dis_sto_MPC} is directly influenced by the tuning of the parameter $\gamma$,
which determines the degree of temporal aggregation employed throughout the iterative solution process. Fig.~\ref{fig:gamma_sensitivity} presents examples of the convergence behavior over the first four iterations of Algorithm~\ref{alg:dis_sto_MPC} for different values of $\gamma$.
Large values ($\gamma=0.7$) produce highly aggregated models and therefore may require a large number of iterations to converge to the prescribed optimality gap ($\epsilon^{\mathrm{thr}}=1\%$).
Conversely, small values ($\gamma=0.5$) accelerate convergence by generating more accurate aggregated models,
albeit at the expense of increased computational effort per iteration due to their higher temporal resolution.
An intermediate value ($\gamma=0.6$) provides the best trade-off between the number of iterations required for convergence and the computational complexity of the resulting temporally aggregated models.

Within Algorithm~\ref{alg:dis_sto_MPC}, the aggregated model is solved using consensus ADMM. 
An example of the ADMM convergence behavior for different initial values of the step size parameter $\rho$ is shown in Fig.~\ref{fig:ADMM_convergence}. 
A small initial step size ($\rho^0=2$) places less emphasis on constraint enforcement, enabling aggressive primal updates that introduce transient overshoots. 
Conversely, a large initial step size ($\rho^0=10$) excessively prioritizes constraint satisfaction, resulting in conservative updates and slower overall convergence. 
An intermediate value ($\rho^0=4$) achieves a favorable balance between primal progress and constraint feasibility, yielding the fastest convergence. 
This behavior highlights the well-known sensitivity of the ADMM performance to parameter tuning \cite{Boyd}.

\begin{figure}[t]
\centering
\includegraphics[scale=.6]{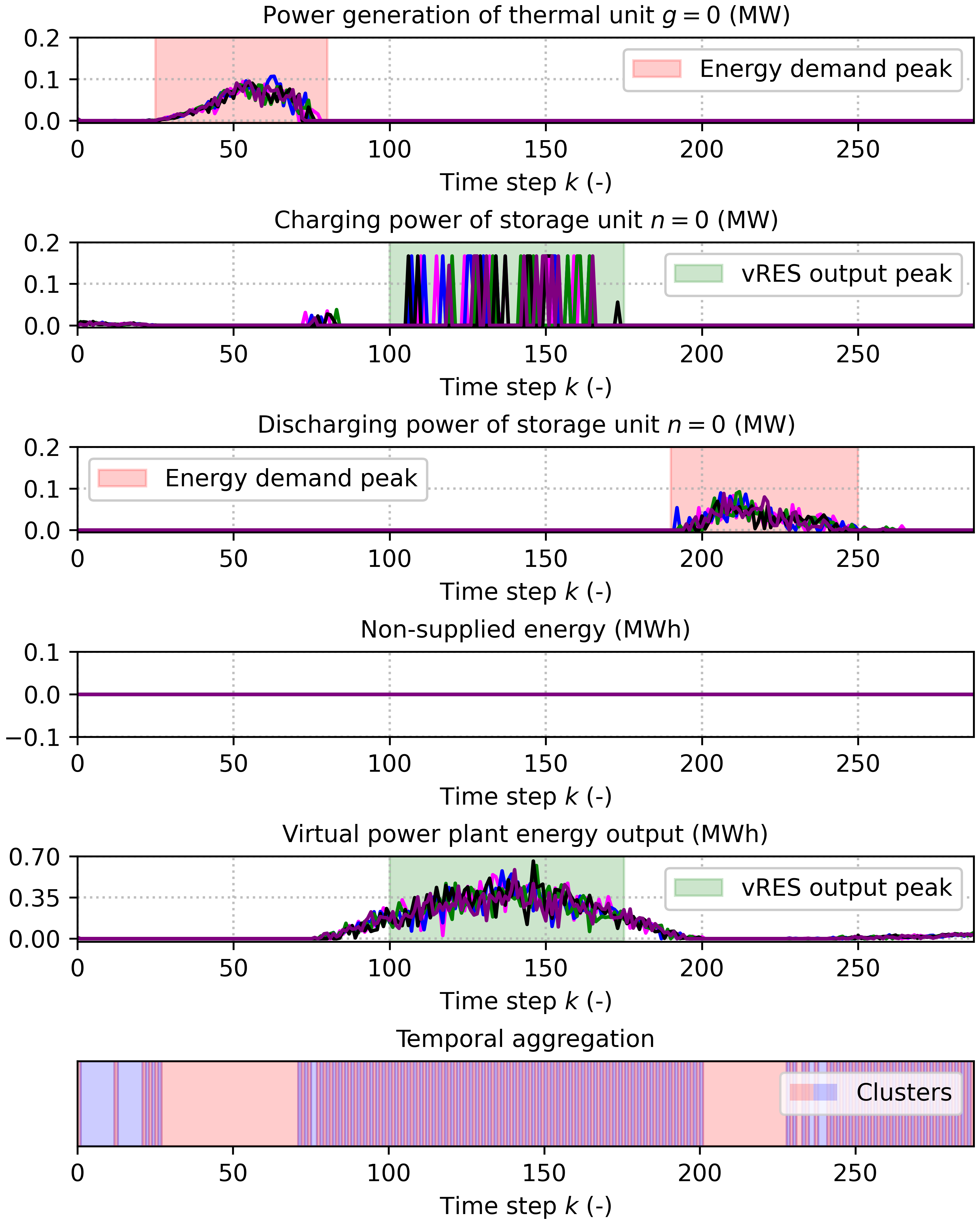}
\caption{Example of the identified clusters and the corresponding energy dispatch obtained using Algorithm~\ref{alg:dis_sto_MPC} for the first storage unit ($n=0$) and the first thermal unit ($g=0$) of the VPP, shown for $5$ randomly selected scenarios from a set with $50$ scenarios. The curves in the first five subplots represent the dispatch trajectories corresponding to the selected scenarios.}
\label{fig:dispatch_example}
\end{figure}

An example of VPP dispatch obtained using Algorithm~\ref{alg:dis_sto_MPC} is shown in Fig.~\ref{fig:dispatch_example} for an MPC iteration starting at midnight.
Despite the simultaneous temporal aggregation and asset-scenario decomposition,
our controller recovers the full-scale optimal control actions, dispatching the VPP flexibility units to meet energy demand with zero non-supplied energy.
Notably, the identified clusters ($197$ in total) effectively capture the periods of peak thermal generation and storage discharge while achieving a $32\%$ reduction in the temporal dimension of the dispatch problem, originally defined over $288$ time steps.

\begin{table}[t]
\caption{
Computational performance of Algorithm~\ref{alg:dis_sto_MPC} in comparison with full-scale MPC (FS-MPC) as $|\boldsymbol{\Omega}|$, $|\boldsymbol{N}|$ and $|\boldsymbol{G}|$ increase. Relative runtime differences are shown in \textbf{bold} in brackets, and intractable FS-MPC instances are marked with the symbol $\dagger$.}
\centering
\setlength{\tabcolsep}{1.7pt}
\begin{tabular}{|c|c|c|c|c|c|c|c|}
\hline
\multirow{2}{*}{$|\boldsymbol{\Omega}|$} & \multirow{2}{*}{$|\boldsymbol{N}|$} 
& \multirow{2}{*}{$|\boldsymbol{G}|$} 
& \textbf{FS-MPC} 
& \multicolumn{4}{c|}{\textbf{Algorithm~\ref{alg:dis_sto_MPC}}} \\
&&& \textbf{Avg. runtime} (s) 
& \multicolumn{1}{c}{\textbf{$\epsilon^{\mathrm{thr}}$}}
& \multicolumn{1}{c}{\textbf{Avg. $\epsilon^\star$}} 
& \multicolumn{1}{c}{\textbf{Avg.} $|\boldsymbol{R}|$} 
& \multicolumn{1}{c|}{\textbf{Avg. runtime} (s)} \\
\hline
\multirow{ 2}{*}{$50$} & \multirow{ 2}{*}{$50$} & \multirow{ 2}{*}{$50$} & \multirow{ 2}{*}{$46.9$} & $1\%$ & $0.7\%$ & $198.8$ & $37.1$ ($\boldsymbol{-21\%}$)\\
& & & & $5\%$ & $3.9\%$ & $141.1$ & $24.6$ ($\boldsymbol{-48\%}$)\\
\hline
\multirow{ 2}{*}{$100$} & \multirow{ 2}{*}{$50$} & \multirow{ 2}{*}{$50$} & \multirow{ 2}{*}{$122.6$} & $1\%$ & $0.8\%$ & $199.0$ & $92.1$ ($\boldsymbol{-25\%}$)\\
& & & & $5\%$ & $4.4\%$ & $138.6$ & $57.9$ ($\boldsymbol{-53\%}$)\\
\hline
\multirow{ 2}{*}{$50$} & \multirow{ 2}{*}{$100$} & \multirow{ 2}{*}{$100$} & \multirow{ 2}{*}{$148.2$} & $1\%$ & $0.3\%$ & $201.6$ & $157.3$ ($\boldsymbol{+6\%}$)\\
& & & & $5\%$ & $3.6\%$ & $148.1$ & $130.9$ ($\boldsymbol{-12\%}$)\\
\hline
\multirow{ 2}{*}{$50$} & \multirow{ 2}{*}{$200$} & \multirow{ 2}{*}{$100$} & \multirow{ 2}{*}{$356.5^{\dagger}$} & $1\%$ & $0.7\%$ & $212.3$ & $276.1$ ($\boldsymbol{-23\%}$)\\
& & & & $5\%$ & $4.3\%$ & $166.9$ & $189.7$ ($\boldsymbol{-47\%}$)\\
\hline
\multirow{ 2}{*}{$50$} & \multirow{ 2}{*}{$100$} & \multirow{ 2}{*}{$200$} & \multirow{ 2}{*}{$388.6^{\dagger}$} & $1\%$ & $0.8\%$ & $200.4$ & $274.2$ ($\boldsymbol{-29\%}$)\\
& & & & $5\%$ & $4.7\%$ & $148.7$ & $186.3$ ($\boldsymbol{-52\%}$)\\
\hline
\multirow{ 2}{*}{$200$} & \multirow{ 2}{*}{$50$} & \multirow{ 2}{*}{$50$} & \multirow{ 2}{*}{$433.8^{\dagger}$} & $1\%$ & $0.9\%$ & $202.8$ & $205.5$ ($\boldsymbol{-53\%}$)\\
& & & & $5\%$ & $4.7\%$ & $108.0$ & $109.1$ ($\boldsymbol{-75\%}$)\\
\hline
\end{tabular}
\label{tab:scenario_assets_scalability}
\end{table}

\begin{table}[t]
\caption{
Computational performance of Algorithm~\ref{alg:dis_sto_MPC}, a priori TSA (AP-TSA), marginal-cost-based TSA (MC-TSA), and consensus ADMM compared with full-scale MPC (FS-MPC) for increasing $|\boldsymbol{K}|$ and $|\boldsymbol{\Omega}|$.
Relative runtime differences for all benchmark methods are computed with respect to FS-MPC, 
with relative runtime reductions for Algorithm~\ref{alg:dis_sto_MPC} highlighted in \textbf{bold}. 
Intractable FS-MPC instances are marked with the symbol $\dagger$.
}
\centering
\setlength{\tabcolsep}{2.5pt}
\begin{tabular}{|c|c|c|c|c|c|c|}
\hline
\multirow{ 2}{*}{$|\boldsymbol{K}|$} &
\multirow{ 2}{*}{$|\boldsymbol{\Omega}|$} &
\textbf{FS-MPC} & 
\multicolumn{4}{c|}{\textbf{Relative difference in average runtime}}\\
 & &
\textbf{Avg. runtime} (s) &
\multicolumn{1}{c}{\textbf{AP-TSA}} &
\multicolumn{1}{c}{\textbf{MC-TSA}} &
\multicolumn{1}{c}{\textbf{ADMM}} &
\multicolumn{1}{c|}{\textbf{Algorithm~\ref{alg:dis_sto_MPC}}}\\
\hline
$288$ & $50$
& $46.9$
& $+31 \%$
& $-9 \%$
& $-2 \%$
& $\boldsymbol{-21\%}$ \\
\hline
$288$ & $100$
& $122.6$
& $+29 \%$
& $-11 \%$
& $-5 \%$
& $\boldsymbol{-25\%}$ \\
\hline
$576$ & $50$
& $114.4$
& $+26 \%$
& $-18 \%$
& $-7 \%$
& $\boldsymbol{-34\%}$ \\
\hline
$576$ & $100$
& $341.7^\dagger$
& $+22 \%$
& $-23 \%$
& $-16 \%$
& $\boldsymbol{-49\%}$ \\
\hline
\end{tabular}
\label{tab:benchmark_methods_runtimes}
\end{table}

Table~\ref{tab:scenario_assets_scalability} reports the average optimality gap $\epsilon^\star$ attained at convergence, the average number of clusters $|\boldsymbol{R}|$, and the average runtime of Algorithm~\ref{alg:dis_sto_MPC} over the simulated MPC iterations,
compared with full-scale MPC, as the optimality gap threshold $\epsilon^{\mathrm{thr}}$ and the number of assets and scenarios vary.
The reported results account for all the algorithmic steps involved in Algorithm~\ref{alg:dis_sto_MPC}.
For small problem instances, Algorithm~\ref{alg:dis_sto_MPC} may underperform relative to full-scale MPC (see the third row),
yet it reduces runtime by up to $53\%$ for $\epsilon^{\mathrm{thr}} = 1\%$ and $75\%$ for $\epsilon^{\mathrm{thr}} = 5\%$ as the problem size grows.
Crucially, it recovers tractability where full-scale MPC fails to compute the dispatch within the available control time ($5$~min).

Finally, Table~\ref{tab:benchmark_methods_runtimes} compares the average runtime across the simulated MPC iterations for full-scale MPC, traditional a priori TSA, marginal-cost-based TSA (without asset-scenario decomposition), consensus ADMM (without temporal aggregation), and Algorithm~\ref{alg:dis_sto_MPC},
across dispatch problems with $|\boldsymbol{N}| = |\boldsymbol{G}| = 50$ and increasing prediction horizon length and number of scenarios.
The TSA methods and ADMM employ the parameter settings adopted for Algorithm~\ref{alg:dis_sto_MPC}.
All methods are terminated upon reaching a $1\%$ optimality gap.
The a priori TSA method relies on energy prices as clustering features, while deriving objective function bounds through the same iterative procedure adopted for Algorithm~\ref{alg:dis_sto_MPC}.

As discussed in Subsection~\ref{sub:Results_MarginalCost}, the use of energy prices as clustering features results in poor aggregation performance, requiring several computationally intensive refinement iterations before convergence and ultimately resulting in higher computational effort than full-scale MPC.
In contrast, the proposed marginal-cost-based TSA method significantly reduces computational complexity relative to full-scale MPC by deriving temporally aggregated models that accurately retain the active constraints governing the dispatch problem, as illustrated in Subsection~\ref{sub:Results_MarginalCost}.
Similarly, consensus ADMM enables significant computational savings, albeit through asset-scenario decomposition rather than temporal aggregation. 
As expected, the proposed Algorithm~\ref{alg:dis_sto_MPC} provides the greatest computational benefits,
particularly as the temporal horizon and the number of scenarios in the dispatch problem increase,
due to the synergistic effect of simultaneously exploiting temporal aggregation and asset-scenario decomposition.

\section{Conclusion and Future Work}
\label{sec:Conclusion}
This paper addresses the real-time energy dispatch of a VPP formulated as a stochastic MPC problem.
Owing to multiple sources of uncertainty, nonlinear dynamics, and several coupling constraints, solving the problem at full scale rapidly becomes intractable as the dispatch horizon, number of scenarios, and number of DERs increase.
To overcome this limitation, we propose a novel control scheme integrating MPC with TSA and distributed optimization to simultaneously reduce the temporal, asset, and scenario dimensions of the problem.
Notably, the controller employs a closed-loop, a posteriori TSA method that exploits dual information, specifically marginal cost estimates,
to refine temporal aggregation, in contrast to traditional TSA methods that rely solely on input time series analysis.
Crucially, the controller embeds a rigorous performance guarantee in the form of theoretically validated bounds on its approximation error relative to full-scale MPC.
Numerical results show that the proposed controller reduces runtime by up to $53\%$ relative to full-scale MPC while maintaining over $99\%$ accuracy, with runtime reductions reaching $75\%$ when the accuracy requirement is relaxed to $95\%$.
Crucially, it recovers tractability where full-scale MPC fails to compute the dispatch within the available control time.

Future research directions include extending the proposed methodology to dispatch formulations beyond the considered stochastic QCQP problem \eqref{FS_SMPC}.
In particular, generalizing the proposed methodology to nonconvex formulations of the dispatch problem,
including unit commitment decisions, startup costs, reserve constraints, as well as network and power flow constraints,
constitutes a primary avenue for future research.


\appendix[Proof of Proposition~\ref{prop:main_result}]

This appendix presents the proof of Proposition~\ref{prop:main_result} from Subsection~\ref{subsec:TheoreticalResult}.

\begin{proof}
We first demonstrate that any $\boldsymbol{\bar{z}}$ obtained through \eqref{prop:main_result_eq1}--\eqref{prop:main_result_eq6} is feasible for the temporally aggregated model \eqref{AGG_SMPC}.

Using \eqref{prop:main_result_eq4}, we reformulate \eqref{AGG_SMPC:thr_ramp1} and \eqref{AGG_SMPC:thr_ramp2} as
\begin{align}
    & \bar{p}^\mathrm{th}_{g,\omega,r+1} - \bar{p}^\mathrm{th}_{g,\omega,r} \, K_r = \sum_{k \in \boldsymbol{K}_{r+1}} \frac{p^{\mathrm{th}}_{g,\omega,k}}{K_{r+1}} - \sum_{k \in \boldsymbol{K}_r} p^{\mathrm{th}}_{g,\omega,k} \nonumber\\
    & \leq  p^{\mathrm{th}}_{g,\omega,\min(\boldsymbol{K}_{r+1})} + R^\mathrm{th}_g \, \frac{K_{r+1} - 1}{2} - p^{\mathrm{th}}_{g,\omega,\max(\boldsymbol{K}_{r})} \nonumber\\
    & \leq R^\mathrm{th}_g + R^\mathrm{th}_g \, \frac{K_{r+1} - 1}{2}, \; \forall g, \forall \omega, \forall r \in \boldsymbol{R} \setminus \{R-1\}, \label{prop:agg_thr_ramp1}
\end{align}
and
\begin{align}
    & \bar{p}^\mathrm{th}_{g,\omega,r} - \bar{p}^\mathrm{th}_{g,\omega,r+1} \, K_{r+1} =  \sum_{k \in \boldsymbol{K}_r} \!\! \frac{p^{\mathrm{th}}_{g,\omega,k}}{K_r} - \sum_{k \in \boldsymbol{K}_{r+1}} \!\! p^{\mathrm{th}}_{g,\omega,k} \nonumber\\
    & \leq  p^{\mathrm{th}}_{g,\omega,\max(\boldsymbol{K}_{r})} + R^\mathrm{th}_g \, \frac{K_{r} - 1}{2} - p^{\mathrm{th}}_{g,\omega,\min(\boldsymbol{K}_{r+1})} \nonumber\\
    & \leq R^\mathrm{th}_g + R^\mathrm{th}_g \, \frac{K_r - 1}{2}, \; \forall g, \forall \omega, \forall r \in \boldsymbol{R} \setminus \{R-1\}, \label{prop:agg_thr_ramp2}
\end{align}
respectively.

Using \eqref{prop:main_result_eq4}, we reformulate \eqref{AGG_SMPC:thr_emission} as
\begin{align}
    \label{prop:agg_emissions}
    & \frac{\alpha_g}{K_r} \! \left( \sum_{k \in \boldsymbol{K}_r} \!\! p^\mathrm{th}_{g,\omega,k} \! \right)^2
    \!\!-\! \alpha_g \left( K_r \!-\! 1 \right) \! \left( \overline{P}^\mathrm{th}_g \right)^2
    \!\!+\! \frac{\beta_g}{K_r} \! \sum_{k \in \boldsymbol{K}_r} \!\! p^\mathrm{th}_{g,\omega,k} \nonumber \\
    & \leq
    \!\!\! \sum_{k \in \boldsymbol{K}_r} \frac{\alpha_g \! \left( p^\mathrm{th}_{g,\omega,k} \right)^2 \!\!+\! \beta_g \, p^\mathrm{th}_{g,\omega,k}}{K_r}
    \!\leq\! L^{\mathrm{CO}_2}_g, \; \forall g, \! \forall \omega, \! \forall r.
\end{align}

The full-scale constraints \eqref{model:thr_ramp1}, \eqref{model:thr_ramp2}, and \eqref{model:thr_emission}
imply the aggregated constraints \eqref{prop:agg_thr_ramp1}, \eqref{prop:agg_thr_ramp2}, and \eqref{prop:agg_emissions}, respectively,
since they are enforced over the set $\boldsymbol{K}_r$ rather than over the individual time steps $k$.
Similarly, constraints \eqref{FS_SMPC:bal}, \eqref{model:storage_soc}--\eqref{model:thr_bounds}, and \eqref{FS_SMPC:non_ant}
imply \eqref{AGG_SMPC:bal}--\eqref{AGG_SMPC:thr_bounds} and \eqref{AGG_SMPC:non_ant}, as demonstrated in \cite{santosuosso2025we}.
It follows that any $\boldsymbol{\bar{z}}$ obtained from a feasible solution $\boldsymbol{z}$ of the full-scale model \eqref{FS_SMPC} via \eqref{prop:main_result_eq1}--\eqref{prop:main_result_eq6} is feasible for \eqref{AGG_SMPC}.

Substituting $\boldsymbol{\bar{z}}$ into $\bar{F}(\boldsymbol{\bar{z}})$, as defined in \eqref{AGG_SMPC:obj}, yields
\begin{align}
\bar{F}(\boldsymbol{\bar{z}}) = 
& \sum_{\omega \in \boldsymbol{\Omega}} \phi_{\omega} \sum_{r \in \boldsymbol{R}} \left(
- \max_{k \in \boldsymbol{K}_r}\left\{\pi_{\omega,k}\right\} \sum_{k \in \boldsymbol{K}_r} e_{\omega,k} \right) \nonumber\\
& + \sum_{\omega \in \boldsymbol{\Omega}} \phi_{\omega} \sum_{r \in \boldsymbol{R}} \sum_{k \in \boldsymbol{K}_r} C^{\mathrm{ns}} e^{\mathrm{ns}}_{\omega,k} \nonumber\\
& + \sum_{\omega \in \boldsymbol{\Omega}} \phi_{\omega} \sum_{g \in \boldsymbol{G}} \sum_{r \in \boldsymbol{R}} \sum_{k \in \boldsymbol{K}_r} C^{\mathrm{th}}_g p^{\mathrm{th}}_{g,\omega,k} \Delta \nonumber\\
& + \sum_{\omega \in \boldsymbol{\Omega}} \phi_{\omega} \sum_{n \in \boldsymbol{N}} \sum_{r \in \boldsymbol{R}} \sum_{k \in \boldsymbol{K}_r} \left( C^{\mathrm{c}}_n p^{\mathrm{c}}_{n,\omega,k} + C^{\mathrm{d}}_n p^{\mathrm{d}}_{n,\omega,k} \right) \Delta \nonumber\\
& + \sum_{\omega \in \boldsymbol{\Omega}} \phi_{\omega} \sum_{n \in \boldsymbol{N}} \sum_{r \in \boldsymbol{R}} C^{\mathrm{ref}}_n \left( e^{\mathrm{s}}_{n,\omega,\min(\!\boldsymbol{K}_r\!)} - E^{\mathrm{ref}}_n \right)^2 \nonumber\\
\leq
& \sum_{\omega \in \boldsymbol{\Omega}} \phi_{\omega} \sum_{k \in \boldsymbol{K}} \left(
- \pi_{\omega,k} \, e_{\omega,k} \right)
+ \sum_{\omega \in \boldsymbol{\Omega}} \phi_{\omega} \sum_{k \in \boldsymbol{K}} C^{\mathrm{ns}} e^{\mathrm{ns}}_{\omega,k} \nonumber\\
& + \sum_{\omega \in \boldsymbol{\Omega}} \phi_{\omega} \sum_{g \in \boldsymbol{G}} \sum_{k \in \boldsymbol{K}} C^{\mathrm{th}}_g p^{\mathrm{th}}_{g,\omega,k} \Delta \nonumber\\
& + \sum_{\omega \in \boldsymbol{\Omega}} \phi_{\omega} \sum_{n \in \boldsymbol{N}} \sum_{k \in \boldsymbol{K}} \left( C^{\mathrm{c}}_n p^{\mathrm{c}}_{n,\omega,k} + C^{\mathrm{d}}_n p^{\mathrm{d}}_{n,\omega,k} \right) \Delta \nonumber\\
& + \sum_{\omega \in \boldsymbol{\Omega}} \!\phi_{\omega}\!\! \sum_{n \in \boldsymbol{N}} \sum_{r \in \boldsymbol{R}} \!C^{\mathrm{ref}}_n \!\!\left( e^{\mathrm{s}}_{n,\omega,\min(\!\boldsymbol{K}_r\!)} - E^{\mathrm{ref}}_n \right)^2,
\label{prop:main_result_obj1}
\end{align}
where the inequality follows directly from the fact that
\begin{equation*}
\max_{k \in \boldsymbol{K}_r}\{\pi_{\omega,k}\} \geq \pi_{\omega,k},
\; \forall \omega, \forall k \in \boldsymbol{K}_r, \forall r.
\end{equation*}

Furthermore, the last term in \eqref{FS_SMPC:obj} can be written as
\begin{multline}
\sum_{\omega \in \boldsymbol{\Omega}} \phi_{\omega} \sum_{n \in \boldsymbol{N}} \sum_{k \in \boldsymbol{K}} C^{\mathrm{ref}}_n \left(e^{\mathrm{s}}_{n,\omega,k} - E^{\mathrm{ref}}_n\right)^2 =\\
\sum_{\omega \in \boldsymbol{\Omega}} \phi_{\omega} \sum_{n \in \boldsymbol{N}} \sum_{r \in \boldsymbol{R}} \sum_{k \in \boldsymbol{K}_r} C^{\mathrm{ref}}_n \left(e^{\mathrm{s}}_{n,\omega,k} - E^{\mathrm{ref}}_n\right)^2.
\label{prop:main_result_obj2}
\end{multline}

In addition, the following holds:
\begin{multline}
\sum_{k \in \boldsymbol{K}_r} C^{\mathrm{ref}}_n \left(e^{\mathrm{s}}_{n,\omega,k} - E^{\mathrm{ref}}_n\right)^2 = C^{\mathrm{ref}}_n \left( e^{\mathrm{s}}_{n,\omega,\min(\!\boldsymbol{K}_r\!)} - E^{\mathrm{ref}}_n \right)^2\\
+ \sum_{k \in \boldsymbol{K}_r \setminus \left\{\min(\!\boldsymbol{K}_r\!)\right\}} C^{\mathrm{ref}}_n \left(e^{\mathrm{s}}_{n,\omega,k} - E^{\mathrm{ref}}_n\right)^2, \; \forall n, \forall \omega, \forall r.
\label{prop:main_result_obj3}
\end{multline}

From \eqref{prop:main_result_obj2} and \eqref{prop:main_result_obj3}, it follows that
\begin{multline}
\sum_{\omega \in \boldsymbol{\Omega}} \phi_{\omega} \sum_{n \in \boldsymbol{N}} \sum_{k \in \boldsymbol{K}} C^{\mathrm{ref}}_n \left(e^{\mathrm{s}}_{n,\omega,k} - E^{\mathrm{ref}}_n\right)^2 \geq\\
\sum_{\omega \in \boldsymbol{\Omega}} \!\phi_{\omega}\!\! \sum_{n \in \boldsymbol{N}} \sum_{r \in \boldsymbol{R}} \!C^{\mathrm{ref}}_n \!\!\left( e^{\mathrm{s}}_{n,\omega,\min(\!\boldsymbol{K}_r\!)} - E^{\mathrm{ref}}_n \right)^2.
\label{prop:main_result_obj4}
\end{multline}

Finally, combining \eqref{prop:main_result_obj1} and \eqref{prop:main_result_obj4} yields
\begin{align}
\bar{F}(\boldsymbol{\bar{z}}) \leq 
& \sum_{\omega \in \boldsymbol{\Omega}} \phi_{\omega} \sum_{k \in \boldsymbol{K}} \left(
- \pi_{\omega,k} \, e_{\omega,k} \right)
+ \sum_{\omega \in \boldsymbol{\Omega}} \phi_{\omega} \sum_{k \in \boldsymbol{K}} C^{\mathrm{ns}} e^{\mathrm{ns}}_{\omega,k} \nonumber\\
& + \sum_{\omega \in \boldsymbol{\Omega}} \phi_{\omega} \sum_{g \in \boldsymbol{G}} \sum_{k \in \boldsymbol{K}} C^{\mathrm{th}}_g p^{\mathrm{th}}_{g,\omega,k} \Delta \nonumber\\
& + \sum_{\omega \in \boldsymbol{\Omega}} \phi_{\omega} \sum_{n \in \boldsymbol{N}} \sum_{k \in \boldsymbol{K}} \left( C^{\mathrm{c}}_n p^{\mathrm{c}}_{n,\omega,k} + C^{\mathrm{d}}_n p^{\mathrm{d}}_{n,\omega,k} \right) \Delta \nonumber\\
& + \sum_{\omega \in \boldsymbol{\Omega}} \phi_{\omega} \sum_{n \in \boldsymbol{N}} \sum_{k \in \boldsymbol{K}} C^{\mathrm{ref}}_n \left(e^{\mathrm{s}}_{n,\omega,k} - E^{\mathrm{ref}}_n\right)^2 \nonumber\\
=
& \; F(\boldsymbol{z}),
\end{align}
where $F(\boldsymbol{z})$ is defined as in \eqref{FS_SMPC:obj}.
\end{proof}

\section*{Acknowledgments}
Funded by the European Union (ERC, NetZero-Opt, 101116212). Views and opinions expressed are however those of the authors only and do not necessarily reflect those of the European Union or the European Research Council. Neither the European Union nor the granting authority can be held responsible for them.

We thank the anonymous reviewers for their constructive comments and valuable suggestions.

\end{document}